\documentclass[12pt]{amsart}

\usepackage{amsthm}

\usepackage{amssymb}
\usepackage{verbatim}
\usepackage[curve,matrix,arrow]{xy}
\usepackage{amsmath,amsfonts}
\usepackage{graphicx}
\usepackage{epsf}

\newcommand{\thmref}[1]{Theorem~\ref{#1}}

\newcommand{\lemref}[1]{Lemma~\ref{#1}}

\newcommand{\remref}[1]{Remark~\ref{#1}}
\newcommand{\corref}[1]{Corollary~\ref{#1}}

\newcommand{\figref}[1]{Figure~\ref{#1}}

  {\end{list}}

\newtheorem{theorem}{Theorem}[section]

\newtheorem{corollary}[theorem]{Corollary}

\newtheorem{lemma}[theorem]{Lemma}

\newtheorem{remark}[theorem]{Remark}

\newcommand{\secref}[1]{\S\ref{#1}}

\def\mod{\text{mod}}

\begin{document}

\title[Beyond Wolstenholme's Theorem]
{Beyond Wolstenholme's Theorem}

\author{Zubeyir Cinkir}
\address{Zubeyir Cinkir\\
Department of Industrial Engineering\\
Abdullah Gul University\\
38100, Kayseri, TURKEY\\}
\email{zubeyir.cinkir@agu.edu.tr}

\thanks{I would like to thank to anonymous referee for helpful feedback on earlier version of this paper.}

\begin{abstract}
Wolstenholme's type summations involve certain powers of all residues $k$ modulo some prime number $p$. We first consider the sums of double or triple products of certain powers of all residues, e.g., the sums of the terms $(a+k)^m(b+k)^n$ or 
$(a+k)^m(b+k)^n(c+k)^s$ as $k$ ranges over all residues modulo $p$. We consider the sums of double or triple ratios of such terms. We showed that each of such sums is congruent to some simpler expression involving certain binomial coefficients. We also generalize these results to the sums of products or ratios of arbitrary $n$ terms: $(a_1+k)^{m_1}$, ..., $(a_n+k)^{m_n}$. We relate such summations to the sum of certain coefficients of polynomials of type $(a_1-a_n+x)^{m_1} \cdots (a_{n-1}-a_n+x)^{m_{n-1}}$.
\end{abstract}

\maketitle

\section{Introduction}\label{sec introduction}

The simplest of Wolstenholme's type summations is as follows:
\begin{theorem}\label{thm power sum0}
Let $p$ be a prime number. For any non-negative integer $n$, we have the following congruence modulo $p$:
\begin{equation*}\label{eqn sum0}
\begin{split}
\sum_{k=1}^{p-1} k^n \equiv
\begin{cases} 0, & \text{if $p-1 \nmid n$}\\
-1, & \text{if $p-1 \mid n$}.
\end{cases}
\end{split}
\end{equation*}
\end{theorem}
A proof of \thmref{thm power sum0} and references to other proofs can be found in \cite{MS}. 

Wolstenholme's type summations involving slightly different terms can be given as follows:
\begin{theorem}\label{thm power sum1}
Let $p \geq 5$ be a prime number. We have the following congruences modulo $p$:
\begin{equation*}\label{eqn sum1}
\begin{split}
\sum_{k=1}^{p-1} \frac{1}{k} \equiv 0, \quad \quad
\sum_{k=1}^{p-1} \frac{1}{k^2} \equiv 0, \quad \quad
\sum_{k=1}^{p-1} \frac{1}{k^3} \equiv 0.
\end{split}
\end{equation*}
\end{theorem}
In \thmref{thm power sum1}, the first congruence holds modulo $p^2$ if $p \geq 5$ and this is known as Wolstenholme's Theorem. Its proofs can be found in \cite{K}, \cite[Thm 116]{HW} and \cite{W}.
Proofs of the second congruence can be found in  \cite{A}, \cite[Thm 117]{HW} and \cite{M}.
The third congruence holds modulo $p^2$ if $p>5$ \cite[Thm 131]{HW}.

There are various generalizations of \thmref{thm power sum0} and \thmref{thm power sum1}. We list some of them below.
\begin{theorem}\label{thm power sum2}
Let $p$ be a prime number, and let $n$ be a positive integer such that $2n \leq p-1$. We have the following congruences:
\begin{equation*}\label{eqn sum2}
\begin{split}
\sum_{k=1}^{p-1} \frac{1}{k^{2n-1}} \equiv 0 \, \, \, mod \, \, p^2, \quad \quad
\sum_{k=1}^{p-1} \frac{1}{k^{2n}} \equiv 0 \, \, \, mod \, \, p.
\end{split}
\end{equation*}
\end{theorem}
The proof of \thmref{thm power sum2} can be found in \cite[Thm 3]{B}.

The following references contain more of similar results: \cite{A}, \cite{CD}, \cite{G}, \cite{Me} and \cite{ZC}.
One can also check the references in \cite{Me2} for similar results.

We recall Fermat's Little Theorem, as we use it frequently in our proofs \cite[Thm 71]{HW}.
\begin{equation}\label{eqn FLT}
\begin{split}
\text{If $p$ is a prime and $p \nmid a$ then $a^{p-1} \equiv 1 \, \, \, mod \, \, p$.}
\end{split}
\end{equation}
By applying \eqref{eqn FLT}, we have $a^m \equiv a^{m \, \, \mod \, \, p-1} \, \, \, mod \, \, p$ for any
prime number $p$, any integers $a \in \{1, \, 2, \, \ldots, \, p-1\}$ and any positive integer $m$. Therefore, it will be enough to consider the sums of terms with powers less than $p$. This is what we do in this paper. 

For each integer $n$, $k$ and $s$ with $0 \leq s \leq k \leq n$, we have the cancellation identity \cite[pg 29]{BW}:
\begin{equation}\label{eqn product1}
\begin{split}
\binom{n}{k} \binom{k}{s}=\binom{n}{s}\binom{n-s}{k-s}.
\end{split}
\end{equation}

Next, we recall the semi-symmetry properties \cite[Theorems 2.2 and 2.4]{CO}.
Let $p$ be a prime number. For each integers $k$ and $s$ with $0 \leq s \leq k < p$, we have the following congruences modulo $p$:
\begin{equation}\label{eqn thms CO}
\begin{split}
&\binom{k}{s} \equiv (-1)^{k+s} \binom{p-1-s}{k-s} =  (-1)^{k+s} \binom{p-1-s}{p-1-k},\\
&\binom{k}{s} \equiv  (-1)^{s} \binom{p-1-k+s}{s} = (-1)^{s} \binom{p-1-k+s}{p-1-k}.
\end{split}
\end{equation}

\section{The sums of ratios or products of two power terms}\label{sec sumterm2}
In this section, we consider the sums of the ratios of powers of two residue terms of type $\frac{(a+k)^m}{(b+k)^n}$ or the sums of the products of powers of two residue terms of type $(a+k)^m(b+k)^n$. In these sums, the residue terms $(a+k)$ and $(b+k)$ differ by a fixed value. 

We first consider the sums of type $\sum_{\substack{k=1 \\ k \neq a}}^{p-1} \frac{k^m}{(a-k)^n}$, where $p$ is a prime number and $a$, $m$ and $n$ are any non-negative integers less than $p$.
We begin by experimenting with small prime numbers for various values of $a$.
We use \cite{MMA} for our computations. We provide two instances of our computations in 
\figref{fig mod11b} and \figref{fig mod11c}. 

\begin{figure}
\centering
\includegraphics[scale=0.9]{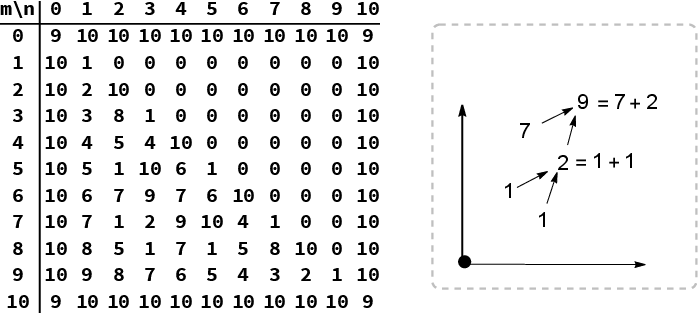} \caption{Summations $\sum_{\substack{k=1 \\ k \neq a}}^{p-1} \frac{k^m}{(a-k)^n}$ modulo $p=11$ for $a=1$. } \label{fig mod11b}
\end{figure}
\begin{figure}
\centering
\includegraphics[scale=0.9]{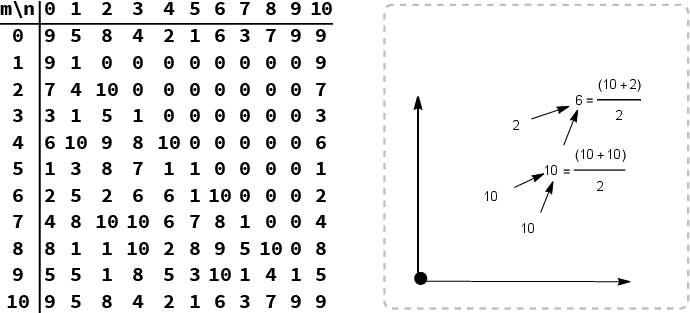} \caption{Summations $\sum_{\substack{k=1 \\ k \neq a}}^{p-1} \frac{k^m}{(a-k)^n}$ modulo $p=11$ for $a=2$.} \label{fig mod11c}
\end{figure}
In these figures, we have the following observations:
\begin{itemize}
\item Numbers at the corners are congruent to $-2 \, \, mod \, \, p$.
\item The first row/column and the last row/column are the same.
\item The numbers in the first row with the order reversed are the same as the numbers at the first column. 
\item In the first column, we have $-a^m \, \, mod \, \, p$ for $m= 1, \, 2, \ldots, p-1$.    
\item Starting from the bottom left corner, we have a modified Pascal's Identity. This is illustrated in both figures. Namely, if the values are considered as a $p \times p$ matrix $B=(b_{i,j})$, then $b_{i,j-1}+b_{i+1,j} \equiv a \cdot b_{i,j} \, \, mod \, \, p$.
\end{itemize}
Based on these observations about our computations, we state the following theorem. Again, its proof is outlined by these observations.
\begin{theorem}\label{thm new power sum1}
Let $p$ be a prime number. For any integers $1 \leq a \leq p-1$, $0 \leq m \leq p-1$ and $0 \leq n \leq p-1$, we have the following congruence modulo $p$:
\begin{equation*}\label{eqn new sum1}
\begin{split}
\sum_{\substack{k=1 \\ k \neq a}}^{p-1} \frac{k^m}{(a-k)^n} \equiv
\begin{cases} 
-a^{p-1-n}, & \text{if $m=0$ and $n \in \{1, \, 2, \, \ldots, \, p-2\}$}\\
-a^m, &  \text{if $n=p-1$ and $m \in \{1, \, 2, \, \ldots, \, p-2\}$} \\
-2,  &\text{if $m, \, n \in \{ 0, \, p-1 \}$}\\
(-1)^{n+1}a^{m-n} \binom{m}{n}, & \text{otherwise}.
\end{cases}
\end{split}
\end{equation*}
\end{theorem}
\begin{proof}
Let $u_{m,n}=\sum_{\substack{k=1 \\ k \neq a}}^{p-1} \frac{k^m}{(a-k)^n}$.

We first consider the case $n=0$. 
\begin{equation*}\label{eqn new sum1a}
\begin{split}
u_{m,0}=&\sum_{\substack{k=1 \\ k \neq a}}^{p-1} k^m= - a^m+\sum_{k=1}^{p-1} k^m  \equiv
\begin{cases} -a^m \, \, \, mod \, \, p, & \text{if $p-1 \nmid m$}\\
-1-a^m\, \, \, mod \, \, p, & \text{if $p-1 \mid m$}
\end{cases}, \quad \text{by \thmref{thm power sum0}}.
\end{split}
\end{equation*}
Therefore, by \eqref{eqn FLT}
\begin{equation}\label{eqn new sum1aa}
\begin{split}
u_{m,0} \equiv
\begin{cases} 
-a^m\, \, \, mod \, \, p, & \text{if $m \in \{1, \, 2, \, \ldots, \, p-2\}$}\\
-2 \, \, \, mod \, \, p, & \text{if $m=0$ or $m=p-1$}\\
\end{cases}.
\end{split}
\end{equation}
Since $\{ a-k \mid k=1, \, 2, \, \ldots, \, p-1 \text{ and $k \neq a$} \} \equiv \{ k \mid k=1, \, 2, \, \ldots, \, p-1 \text{ and $k \neq a$} \} \, \, \, mod \, \, p$, we have
\begin{equation}\label{eqn new sum1aaa}
\begin{split}
u_{0,n} \equiv u_{p-1-n,0} \, \, \, mod \, \, p.
\end{split}
\end{equation}
By \eqref{eqn FLT}, for any integers $m$ and $n$ we have 
\begin{equation}\label{eqn new sum1b}
\begin{split}
u_{p-1,n} \equiv u_{0,n} \, \, \, mod \, \, p \quad \text{and} \quad u_{m,p-1} \equiv u_{m,0} \, \, \, mod \, \, p.
\end{split}
\end{equation}
In particular, we have $u_{p-1,p-1} \equiv u_{0,0} \equiv -2 \, \, \, mod \, \, p$, $u_{0,p-1} \equiv u_{0,0} \equiv -2 \, \, \, mod \, \, p$,  $u_{p-1,0}  \equiv  u_{0,0}   \equiv -2$. Moreover, by \eqref{eqn new sum1aaa} and \eqref{eqn new sum1aa}
\begin{equation}\label{eqn new sum1bb}
\begin{split}
u_{0,n} \equiv
\begin{cases} 
-a^{p-1-n}\, \, \, mod \, \, p, & \text{if $n \in \{1, \, 2, \, \ldots, \, p-2\}$}\\
-2 \, \, \, mod \, \, p, & \text{if $n=0$ or $n=p-1$}\\
\end{cases}.
\end{split}
\end{equation}
Note that \eqref{eqn new sum1b} is consistent with the congruences given in the theorem. Namely, we have the following congruences modulo $p$:

$u_{p-1,n}  \equiv  (-1)^{n+1}a^{p-1-n} \binom{p-1}{n} \equiv -a^{p-1-n} \equiv u_{0,n}$ for each $n \in \{1, \ 2, \, \ldots, \, p-2\}$.

$u_{m,p-1}  \equiv -a^{m} \equiv (-1)^{0+1}a^{m-0} \binom{m}{0} \equiv u_{m,0}$ for each $m \in \{1, \ 2, \, \ldots, \, p-2\}$.

To complete the proof of the theorem, we only need to consider the cases $1 \leq m \leq p-2$ and $1 \leq n \leq p-2$. 
We prove these cases by induction on $m$. For such values of $m$ and $n$ we need to show that $u_{m,n} \equiv (-1)^{m+n}a^{m-n} \binom{m}{n}$. We first need an identity which is a modified version of Pascal's Identity of binomial coefficients.

\textbf{Claim:} For any non-negative integers $n$ and $m$ we have 
\begin{equation}\label{eqn new sum1c}
\begin{split}
u_{m,n}+u_{m+1,n+1}=a \cdot u_{m,n+1}.
\end{split}
\end{equation}
\textbf{Proof of the Claim:} By the definition of $u_{m,n}$,
\begin{equation*}\label{eqn claim1}
\begin{split}
u_{m,n}+u_{m+1,n+1}&=\sum_{\substack{k=1 \\ k \neq a}}^{p-1} \frac{k^m}{(a-k)^n}+\sum_{\substack{k=1 \\ k \neq a}}^{p-1} \frac{k^{m+1}}{(a-k)^{n+1}}=\sum_{\substack{k=1 \\ k \neq a}}^{p-1} \frac{(a-k)k^m+k^{m+1}}{(a-k)^{n+1}}\\
&=a \sum_{\substack{k=1 \\ k \neq a}}^{p-1} \frac{k^{m}}{(a-k)^{n+1}}=a \cdot u_{m,n+1}.
\end{split}
\end{equation*}

Setting $m=0$ in \eqref{eqn new sum1c} gives $u_{1,n+1}=a \cdot u_{0,n+1}-u_{0,n}$ for each integer $n \in \{0, \ 1, \, \ldots, \, p-3\}$. Thus,  we use \eqref{eqn new sum1bb} to obtain $u_{1,n+1} \equiv a \cdot (-a^{p-n-2})+a^{p-n-1}=0$ for $1 \leq n \leq p-3$ and $u_{1,1} = a \cdot u_{0,1}-u_{0,0} \equiv a(-a^{p-2})+2 \equiv 1$ for $n=0$. On the other hand, 
$(-1)^{n+1}a^{1-n} \binom{1}{n} \equiv 0$ if  $2 \leq n \leq p-2$, and $(-1)^{n+1}a^{1-n} \binom{1}{n} \equiv 1$ if  $n=1$.
Thus, when $m=1$, we obtain $u_{m,n} \equiv (-1)^{n+1}a^{m-n} \binom{m}{n}$ for each  $1 \leq n \leq p-2$. Here, we have congruences modulo $p$.

We assume that $u_{m,n} \equiv (-1)^{n+1}a^{m-n} \binom{m}{n} \, \, \, mod \, \, p$ for each  $1 \leq n \leq p-2$ holds for every integer $m=j$ with $1 \leq j \leq p-3$. Next, we prove this for $m=j+1$.
\begin{equation*}\label{eqn new sum2a0}
\begin{split}
u_{j+1,1} &=  a \cdot u_{j,1}-u_{j,0}, \quad \text{by \eqref{eqn new sum1c}} \\ 
&\equiv a \cdot \big( a^{j-1} \binom{j}{1} \big) -u_{j,0} \, \, \, mod \, \, p, \quad \text{by the induction assumption}\\
&\equiv a^{j}\binom{j}{1}+a^{j} \binom{j}{0}\, \, \, mod \, \, p, \quad \text{by \eqref{eqn new sum1aa}}\\
&=a^{j} \binom{j+1}{1} \quad \text{by Pascal's Identity}.
\end{split}
\end{equation*}
Let $n \in \{1, \ 2, \, \ldots, \, p-3\}$.

\begin{equation*}\label{eqn new sum2a}
\begin{split}
u_{j+1,n+1} & =  a \cdot u_{j,n+1}-u_{j,n}, \quad \text{by \eqref{eqn new sum1c}} \\
& \equiv a \cdot \big( (-1)^{n+2} a^{j-n-1} \binom{j}{n+1} \big) -(-1)^{n+1} a^{j-n} \binom{j}{n}\, \, \, mod \, \, p, \quad \text{by the assumptions} \\
& = (-1)^{n+2} a^{j-n}\big(  \binom{j}{n+1}+\binom{j}{n} \big)\\
& = (-1)^{n+2} a^{j-n} \binom{j+1}{n+1}  \quad \text{by Pascal's Identity}.
\end{split}
\end{equation*}
Thus, we proved what we wanted for $m=j+1$. Hence, $u_{m,n} \equiv (-1)^{n+1}a^{m-n} \binom{m}{n} \, \, \, mod \, \, p$ for each  $1 \leq n \leq p-2$ holds for each $1 \leq m \leq p-2$ by the induction principle.
\end{proof}

\begin{remark}\label{rem sum1}
For any integers $a, \, m \in \{0, \, 1, \, \ldots, \, p-1\}$, we have the following congruences:
$$
\sum_{\substack{k=0 \\ k \not \equiv -a, }}^{p-1} (a+k)^m \equiv \sum_{\substack{k=0 \\ k \not \equiv a}}^{p-1} (a-k)^m
\equiv \sum_{k=1}^{p-1} k^m \, \, \, mod \, \, p.
$$
Moreover, if $m \neq 0$,
$$
\sum_{\substack{k=0 \\ k \not \equiv -a}}^{p-1} (a+k)^m \equiv \sum_{k=0}^{p-1} (a+k)^m \, \, \, mod \, \, p. 
$$
\end{remark}
Next, we extend \thmref{thm new power sum1} as follows: 
\begin{theorem}\label{thm new power sum3}
Let $p$ be a prime number. For any integers $a, \, b \in \{0, \, 1, \, 2, \, \ldots, \, p-1\}$ with $a \neq b$ and $m, \, n \in \{0, \, 1, \, \ldots, \, p-1\}$, we have the following congruence modulo $p$:
\begin{equation*}\label{eqn new sum3a}
\begin{split}
\sum_{\substack{k=0 \\ k \not \equiv -a, \, -b}}^{p-1} \frac{(a+k)^m}{(b+k)^n} \equiv
\begin{cases} 
(-1)^{n+1}(a-b)^{p-1-n}, & \text{if $m=0$ and $n \in \{1, \, 2, \, \ldots, \, p-2\}$}\\
-(a-b)^m, &  \text{if $n=p-1$ and $m \in \{1, \, 2, \, \ldots, \, p-2\}$} \\
-2,  &\text{if $m, \, n \in \{ 0, \, p-1 \}$}\\
-(a-b)^{m-n}\binom{m}{n}, & \text{otherwise}.
\end{cases}
\end{split}
\end{equation*}
\end{theorem}
\begin{proof}
If $m \in \{ 0, \, p-1 \}$ and 
$1 \leq n \leq p-2$,
we have the following congruences modulo $p$:
\begin{equation}\label{eqn new sum3b}
\begin{split}
\sum_{\substack{k=0 \\ k \not \equiv -a, \, -b}}^{p-1} \frac{(a+k)^m}{(b+k)^n} 
& \equiv  \sum_{\substack{k=0 \\ k \not \equiv -a, \, -b}}^{p-1} \frac{1}{(b+k)^n}, \quad \text{as $(a+k)^0= 1$  and 
$(a+k)^{p-1}= 1$ by \eqref{eqn FLT}}\\
&\equiv \sum_{\substack{k=0 \\ k \not \equiv -a, \, -b}}^{p-1} (b+k)^{p-1-n} \\
&\equiv -(b-a)^{p-1-n}+\sum_{\substack{k=0 \\ k \not \equiv -b}}^{p-1} (b+k)^{p-1-n}\\
&\equiv -(b-a)^{p-1-n}+\sum_{k=1}^{p-1} k^{p-1-n}, \quad \text{by \remref{rem sum1}} \\
& \equiv -(b-a)^{p-1-n}, \quad \text{by \thmref{thm power sum0}}
\end{split}
\end{equation}
Since $-(b-a)^{p-1-n}=(-1)^{n+1}(a-b)^{p-1-n}$, the first congruence in the theorem holds. Moreover,
$(-1)^{n+1}(a-b)^{p-1-n}=-(a-b)^{p-1-n}\binom{p-1}{n}$. Thus, when $m=p-1$ and $n \in \{1, \, 2, \, \ldots, \, p-2\}$
the last congruence in the theorem holds for such $m$ and $n$.

Suppose $n \in \{0, \, p-1 \}$ and $m \in \{1, \, 2, \, \ldots, \, p-2\}$. Then
\begin{equation*}\label{eqn new sum3c}
\begin{split}
\sum_{\substack{k=0 \\ k \not \equiv -a, \, -b}}^{p-1} \frac{(a+k)^m}{(b+k)^n} & \equiv 
\sum_{\substack{k=0 \\ k \not \equiv -a, \, -b}}^{p-1} (a+k)^m \, \, \, mod \, \, p, \quad \text{since $(b+k)^{p-1} \equiv 1$ by \eqref{eqn FLT} and $(b+k)^{0}=1$}\\
& = -(a-b)^m + \sum_{\substack{k=0 \\ k \not \equiv -a}}^{p-1} (a+k)^m \\
& \equiv -(a-b)^m + \sum_{k=1}^{p-1} k^m \, \, \, mod \, \, p, \quad \text{by \remref{rem sum1}}\\
& \equiv -(a-b)^m \, \, \, mod \, \, p, \quad \text{by \thmref{thm power sum0}}
\end{split}
\end{equation*}
This proves the second congruence in the theorem. Note that when $n=0$ and $m \in \{1, \, 2, \, \ldots, \, p-2\}$
the last congruence in the theorem also holds  for such $m$ and $n$.

If $p \nmid s$ for some integer $s$, then $s^0= 1$ and $s^{p-1} \equiv 1$ by \eqref{eqn FLT}. Therefore, when $m, \, n \in \{ 0, \, p-1 \}$ we have
\begin{equation}\label{eqn new sum3d}
\begin{split}
\sum_{\substack{k=0 \\ k \not \equiv -a, \, -b}}^{p-1} \frac{(a+k)^m}{(b+k)^n} \equiv \sum_{\substack{k=0 \\ k \not \equiv -a, \, -b}}^{p-1} 1 \equiv -2 \, \, \, mod \, \, p.
\end{split}
\end{equation}

For the remaining cases, the following congruences modulo $p$ holds:
\begin{equation*}\label{eqn new sum3e}
\begin{split}
\sum_{\substack{k=0 \\ k \not \equiv -a, \, -b}}^{p-1} \frac{(a+k)^m}{(b+k)^n} & \equiv
\sum_{\substack{k=0 \\ k \not \equiv -b}}^{p-1} \frac{(a+k)^m}{(b+k)^n}\\
& \equiv \sum_{\substack{k=0 \\ k \not \equiv -b}}^{p-1} \frac{1}{(b+k)^n} \Big( a^m +\sum_{i=1}^{m} \binom{m}{i}a^{m-i}k^i \Big), \quad \text{by Binomial Theorem}\\ 
&\equiv a^m \sum_{\substack{k=0 \\ k \not \equiv -b}}^{p-1} \frac{1}{(b+k)^n}+ \sum_{i=1}^{m} \binom{m}{i}a^{m-i}  \sum_{\substack{k=0 \\ k \not \equiv -b}}^{p-1}  \frac{k^i}{(b+k)^n} \\
&\equiv a^m \sum_{\substack{k=0 \\ k \not \equiv -b}}^{p-1} (b+k)^{p-1-n}+  \sum_{i=1}^{m} \binom{m}{i}a^{m-i}  (-1)^n \sum_{\substack{k=0 \\ k \not \equiv -b}}^{p-1}  \frac{k^i}{(-b-k)^n}\\
&\equiv \sum_{i=1}^{m} \binom{m}{i}a^{m-i}  (-1)^n \sum_{\substack{k=0 \\ k \not \equiv -b}}^{p-1}  \frac{k^i}{(-b-k)^n},
\quad \text{by \eqref{eqn new sum3b}}\\
&\equiv \sum_{i=1}^{m} \binom{m}{i}a^{m-i} \big( -(-b)^{i-n} \binom{i}{n} \big),
\quad \text{by \thmref{thm new power sum1}}
\end{split}
\end{equation*}
\begin{equation*}\label{eqn new sum3ee}
\begin{split}
&\equiv -\sum_{i=n}^{m} \binom{m}{i} \binom{i}{n} a^{m-i} (-b)^{i-n}, \quad \text{as $\binom{n}{i}=0$ if $i<n$}\\
&\equiv -\sum_{i=0}^{m-n} \binom{m}{i+n} \binom{i+n}{n} a^{m-n-i} (-b)^{i}\\
&\equiv -\sum_{i=0}^{m-n} \binom{m}{n} \binom{m-n}{i} a^{m-n-i} (-b)^{i},
\quad \text{by the cancellation identity} \\
&\equiv -(a-b)^{m-n}\binom{m}{n}, \quad \text{by Binomial Theorem}
\end{split}
\end{equation*}
This completes the proof of the theorem.
\end{proof}
An easy consequence of \thmref{thm new power sum1} and \thmref{thm new power sum3} can be given as follows:
\begin{corollary}\label{cor new power sum3}
Let $p$ be a prime number, and let  $a, \, b \in \{0, \, 1, \, 2, \, \ldots, \, p-1\}$ with $a \neq b$. 
For any integers $m, \, n \in \{1, \, 2, \, \ldots, \, p-2\}$ with $m < n$, we have 
\begin{equation*}\label{eqn new sum3ac}
\begin{split}
\sum_{\substack{k=1 \\ k \not \equiv a}}^{p-1} \frac{k^m}{(a-k)^n} \equiv 0 \, \, \, mod \, \, p, \quad \quad
\sum_{\substack{k=0 \\ k \not \equiv -a, \, -b}}^{p-1} \frac{(a+k)^m}{(b+k)^n} \equiv 0 \, \, \, mod \, \, p.
\end{split}
\end{equation*}
\end{corollary}

\begin{remark}\label{rem sum2}
Let $a=b$, and let $m, \, n \in \{1, \, 2, \, \ldots, \, p-1\}$.

If $m > n$, then 
\begin{equation}\label{eqn rem sum2a}
\begin{split}
\sum_{\substack{k=0 \\ k \not \equiv -a, \, -b}}^{p-1} \frac{(a+k)^m}{(b+k)^n} \equiv \sum_{\substack{k=0 \\ k \not \equiv -a}}^{p-1} (a+k)^{m-n} \equiv \sum_{k=1}^{p-1} k^{m-n}.
\end{split}
\end{equation}	
If $m < n$, then 
\begin{equation}\label{eqn rem sum2b}
\begin{split}
\sum_{\substack{k=0 \\ k \not \equiv -a, \, -b}}^{p-1} \frac{(a+k)^m}{(b+k)^n} \equiv \sum_{\substack{k=0 \\ k \not \equiv -a}}^{p-1} \frac{1}{(a+k)^{n-m}} \equiv \sum_{\substack{k=0 \\ k \not \equiv -a}}^{p-1} (a+k)^{p-1+m-n} \equiv
\sum_{k=1}^{p-1} k^{p-1+m-n}.
\end{split}
\end{equation}
If $m = n$, then 
\begin{equation}\label{eqn rem sum2c}
\begin{split}
\sum_{\substack{k=0 \\ k \not \equiv -a, \, -b}}^{p-1} \frac{(a+k)^m}{(b+k)^n} \equiv \sum_{\substack{k=0 \\ k \not \equiv -a}}^{p-1} 1 \equiv -1.
\end{split}
\end{equation}																								
\end{remark}
Next, we consider the sums of products of the residue terms $(a+k)^m$ and $k^n$: 
\begin{theorem}\label{thmcor new power sum3}
Let $p$ be a prime number. For any integers $a, \, m, \, n \in \{1, \, 2, \, \ldots, \, p-1\}$, we have the following congruence modulo $p$:
\begin{equation*}\label{eqn thmcor new sum3a}
\begin{split}
\sum_{k=0}^{p-1} (a+k)^m k^n \equiv
\begin{cases} 
-2,  &\text{if $m=n=p-1$}\\
-a^{m+n-p+1}\binom{m}{p-1-n}, & \text{otherwise}.
\end{cases}
\end{split}
\end{equation*}
\end{theorem}
\begin{proof}
We can rewrite the given sum as follows:
\begin{equation}\label{eqn cor new sum3b}
\begin{split}
\sum_{k=0}^{p-1} (a+k)^m k^n \equiv \sum_{\substack{k=1 \\ k \not \equiv -a}}^{p-1} \frac{(a+k)^m}{k^{p-1-n}} \, \, \, mod \, \, p.
\end{split}
\end{equation}
Then the result follows from \thmref{thm new power sum3} as $-a^m=-a^{m+n-p+1}\binom{m}{p-1-n} \, \, \, mod \, \, p$ when $n=p-1$ and $m \in \{1, \, 2, \, \ldots, \, p-2\}$.
\end{proof}
We also have the following congruence along with \eqref{eqn cor new sum3b}:
\begin{equation}\label{eqn cor new sum3c}
\begin{split}
\sum_{k=0}^{p-1} (a+k)^m k^n \equiv \sum_{\substack{k=1 \\ k \not \equiv -a}}^{p-1} \frac{k^n}{(a+k)^{p-1-m}} \, \, \, mod \, \, p.
\end{split}
\end{equation}
This is nothing but $(-1)^{m+n+1}a^{m+n-p+1}\binom{n}{p-1-m}$ by \thmref{thm new power sum3} if not both of $m$ and $n$ are $p-1$.
This, in particular, implies the following result (the cases $m, \, n \in \{ 0, \, p-1 \}$ can be checked directly):
\begin{corollary}\label{cor cong}
Let $p$ be a prime number, and let $m, \, n \in \{0, \, 1, \, \ldots, \, p-1\}$. For $M=m+n-(p-1)$, we have 
$$\binom{m}{M}=\binom{m}{p-1-n} \equiv (-1)^{m+n}\binom{n}{p-1-m} =(-1)^{m+n}\binom{n}{M} \, \, \, mod \, \, p.$$
\end{corollary}
Note that \corref{cor cong} can also be derived from the congruences in \eqref{eqn thms CO}.

An extension of \thmref{thmcor new power sum3} is obtained by following the same strategy used to prove \thmref{thmcor new power sum3}:
\begin{theorem}\label{thmcor new power sum3 second}
Let $p$ be a prime number. For any integers $a, \, b, \, m, \, n \in \{1, \, 2, \, \ldots, \, p-1\}$ with $a \neq b$, we have the following congruence modulo $p$:
\begin{equation*}\label{eqn cor new sum3d}
\begin{split}
\sum_{k=0}^{p-1} (a+k)^m (b+k)^n \equiv
\begin{cases} 
-2,  &\text{if $m=n=p-1$}\\
-(a-b)^{m+n}\binom{m}{p-1-n}, & \text{otherwise}.
\end{cases}
\end{split}
\end{equation*}
\end{theorem}

\begin{remark}\label{rem sum3}
Given $a, \, b, \, m, \, n \in \{1, \, 2, \, \ldots, \, p-1\}$, it follows from \thmref{thmcor new power sum3} and \thmref{thmcor new power sum3 second} that
$$\sum_{k=0}^{p-1} (a+k)^m (b+k)^n \equiv  \sum_{k=0}^{p-1} (a-b+k)^m k^n \, \, \, mod \, \, p.$$
Moreover, for each $c \in \{0, \, 1, \, \ldots, \, p-1\}$ we have
$$\sum_{k=0}^{p-1} (a+k)^m (b+k)^n \equiv \sum_{k=0}^{p-1} (a-c+k)^m (b-c+k)^n \, \, \, mod \, \, p.$$
\end{remark}
We can state some direct consequences of \thmref{thmcor new power sum3} and \thmref{thmcor new power sum3 second} as follows:
\begin{corollary}\label{cor sum of products}
Let $p$ be a prime number. For any integers $a, \, b, \, m, \, n \in \{1, \, 2, \, \ldots, \, p-1\}$ with $a \neq b$, 
we have
\begin{equation*}\label{eqn sum of products cor}
\begin{split}
&\text{if $m+n < p-1$,  \, \,} \sum_{k=0}^{p-1} (a+k)^m k^n \equiv 0  \, \, \, mod \, \, p, 
\quad \quad \sum_{k=0}^{p-1} (a+k)^m (b+k)^n \equiv 0 \, \, \, mod \, \, p,\\
&\text{if $m+n = p-1$,  \, \,} \sum_{k=0}^{p-1} (a+k)^m k^n \equiv -1  \, \, \, mod \, \, p, 
\quad \quad \sum_{k=0}^{p-1} (a+k)^m (b+k)^n \equiv -1 \, \, \, mod \, \, p, \\
&\text{if $m+n = p$,  \, \,} \sum_{k=0}^{p-1} (a+k)^m k^n \equiv -m a  \, \, \, mod \, \, p, 
\quad \quad \sum_{k=0}^{p-1} (a+k)^m (b+k)^n \equiv m(b-a) \, \, \, mod \, \, p.
\end{split}
\end{equation*}
\end{corollary}

\section{The sums of ratios or products of three power terms}\label{sec triple}
In this section, we consider the sums of the terms of types $(a+k)^m (b+k)^n (c+k)^s$,  $\frac{(a+k)^m (b+k)^n} {(c+k)^s}$, $\frac{(a+k)^m}{(b+k)^n (c+k)^s}$ and $\frac{1}{(a+k)^m (b+k)^n(c+k)^s}$. All congruences in this section are considered modulo a prime number $p$, even though we don't state this explicitly.
We compute these sums in more than one ways. In this way, we derive different simplified forms for these sums. As a byproduct, we obtain interesting congruences modulo $p$, such as \thmref{thm triple comp1}, \thmref{thm triple cong1}, \thmref{thm triple comp2}, \thmref{thm triple cong2}, \thmref{thm triple cong3}, \corref{cor triple comp0} and \thmref{thm triple comp3}.

When we compute the sums in this section, our method is to express them in terms of the sums considered in \secref{sec sumterm2} so that we can use the results of that section.
\begin{theorem}\label{thm triple power1}
Let $p$ be a prime number. For any integers $a, \, b, \, m, \, n, \, s \in \{1, \, 2, \, \ldots, \, p-1\}$ with $a \neq b$, we have the following congruences modulo $p$:
\begin{equation*}\label{eqn thm triple1}
\begin{split}
\sum_{k=0}^{p-1} (a+k)^m (b+k)^n k^s \equiv 
\begin{cases} 
0,  &\quad \text{if $m+n+s < p-1$}\\
-I_1,  &\quad \text{if $p-1 \leq m+n+s < 2(p-1)$}\\
-I_2-I_3,  &\quad \text{if $2(p-1) \leq m+n+s < 3(p-1)$}\\
-3,   &\quad \text{if $m=n=s=p-1$},
\end{cases}
\end{split}
\end{equation*}
where $M:=m+n+s-(p-1)$, $R:=m+n+s-2(p-1)$ and 
\begin{equation*}\label{eqn thm triple1a0}
\begin{split}
&I_1:=\sum_{j=0}^{M}\binom{m}{M-j}\binom{n}{j}a^{M-j}b^j=\sum_{j=0}^{M}\binom{m}{j}\binom{n}{M-j}a^{j}b^{M-j},\\
&I_2:=\sum_{j=M-m}^{n}\binom{m}{M-j}\binom{n}{j}a^{M-j}b^j=\sum_{j=M-n}^{m} \binom{m}{j}\binom{n}{M-j}a^{j} b^{M-j},\\
&I_3:=\sum_{j=0}^{R}\binom{m}{R-j}\binom{n}{j}a^{R-j}b^j.
\end{split}
\end{equation*}
\end{theorem}
\begin{proof}
Applying the Binomial Theorem to both of $(a+k)^m$ and $(b+k)^n$ gives
$$\sum_{k=0}^{p-1} (a+k)^m (b+k)^n k^s=\sum_{i=0}^{m}\sum_{j=0}^{n} \binom{m}{i}\binom{n}{j}a^{m-i} b^{n-j} \sum_{k=0}^{p-1} k^{i+j+s}.$$
We use \thmref{thm power sum0} to conclude that the inner sum $\sum_{k=0}^{p-1} k^{i+j+s}$ is $0$ if $(p-1) \nmid (i+j+s)$ and $-1$ otherwise. Since $3 \leq m+n+s \leq 3(p-1)$, $(p-1) \mid (i+j+s)$ means that $i+j+s$ is one of $p-1$, $2(p-1)$ or $3(p-1)$.
The last one happens only if $m=n=s=p-1$, in which case  
$\displaystyle{\sum_{k=0}^{p-1} (a+k)^m (b+k)^n k^s \equiv \sum_{\substack{k=1 \\ k \not \equiv -a, \ -b}}^{p-1}1 \equiv -3}.$ 

Since $i+j+s \leq m+n+s$, $(p-1) \nmid (i+j+s)$ when $m+n+s <p-1$. In this case, we have $\sum_{k=1}^{p-1} k^{i+j+s} \equiv 0$ by \thmref{thm power sum0}. This proves the first congruence in the Theorem.

Suppose $p-1 \leq m+n+s < 2(p-1)$.  In this case, $(p-1) \mid (i+j+s)$ iff $p-1 = i+j+s$, so we take $i=p-1-s-j$. Thus,
\begin{equation*}\label{eqn thm triple1a}
\begin{split}
&\sum_{k=0}^{p-1} (a+k)^m (b+k)^n k^s \equiv -\sum_{j=0}^{n} \binom{m}{p-1-s-j}\binom{n}{j}a^{m-(p-1-s-j)} b^{n-j}\\
&\equiv -\sum_{j=n-M}^{n} \binom{m}{p-1-s-j}\binom{n}{j}a^{m-(p-1-s-j)} b^{n-j}, \quad \text{as $p-1-s-j \leq m$ iff $j \geq n-M$}\\
&\equiv -\sum_{j=n-M}^{n} \binom{m}{m-(p-1-s-j)}\binom{n}{j}a^{m-(p-1-s-j)} b^{n-j}
\end{split}
\end{equation*}
Since $m-(p-1-s-j)=M-(n-j)$,
\begin{equation*}\label{eqn thm triple1asec}
\begin{split}
\sum_{k=0}^{p-1} (a+k)^m (b+k)^n k^s
&\equiv -\sum_{j=n-M}^{n} \binom{m}{M-(n-j)}\binom{n}{n-j}a^{M-(n-j)} b^{n-j}\\
%
&\equiv -\sum_{j=0}^{M} \binom{m}{j}\binom{n}{M-j}a^{j} b^{M-j}, \quad \text{by changing the index.}
\end{split}
\end{equation*}
This proves the second congruence in the theorem.

When $2(p-1) \leq m+n+s < 3(p-1)$, $(p-1) \mid (i+j+s)$ iff $p-1 = i+j+s$ or $2(p-1) = i+j+s$.
Since $i$ can be $p-1-s-j$ or $2(p-1)-s-j$, we have
\begin{equation*}\label{eqn thm triple1b0}
\begin{split}
\sum_{k=0}^{p-1} (a+k)^m (b+k)^n k^s &\equiv -\sum_{j=0}^{n} \binom{m}{p-1-s-j}\binom{n}{j}a^{m-(p-1-s-j)} b^{n-j}\\
& \quad \quad -\sum_{j=0}^{n} \binom{m}{2(p-1)-s-j}\binom{n}{j}a^{m-(2(p-1)-s-j)} b^{n-j}.
\end{split}
\end{equation*}
\textbf{Case 1:} Take $i=2(p-1)-s-j$
\begin{equation*}\label{eqn thm triple1b}
\begin{split}
&\sum_{j=0}^{n} \binom{m}{2(p-1)-s-j}\binom{n}{j}a^{m-(2(p-1)-s-j)} b^{n-j}\\
&\equiv \sum_{j=0}^{n} \binom{m}{m-(2(p-1)-s-j)}\binom{n}{j}a^{m-(2(p-1)-s-j)} b^{n-j}\\
&\equiv \sum_{j=0}^{n} \binom{m}{R-(n-j)}\binom{n}{n-j}a^{R-(n-j)} b^{n-j}, \quad \text{as $m-(2(p-1)-s-j)=R-(n-j)$}\\
&\equiv \sum_{j=n-R}^{n} \binom{m}{R-(n-j)}\binom{n}{n-j}a^{R-(n-j)} b^{n-j}, \quad \text{as $R-(n-j) \geq 0$ iff $j \geq n-R$}\\
&\equiv \sum_{j=0}^{R} \binom{m}{j}\binom{n}{R-j}a^{j} b^{R-j}, \quad \text{by the change of index.}
\end{split}
\end{equation*}
\textbf{Case 2:} Take $i=p-1-s-j$.
\begin{equation*}\label{eqn thm triple1c}
\begin{split}
&\sum_{j=0}^{n} \binom{m}{p-1-s-j}\binom{n}{j}a^{m-(p-1-s-j)} b^{n-j}\\
&= \sum_{j=0}^{n} \binom{m}{m-(p-1-s-j)}\binom{n}{j}a^{m-(p-1-s-j)} b^{n-j}\\
&\equiv \sum_{j=0}^{n} \binom{m}{M-(n-j)}\binom{n}{n-j}a^{M-(n-j)} b^{n-j}, \, \, \, \text{as $m-((p-1)-s-j)=M-(n-j)$}\\
&\equiv \sum_{j=0}^{n+m-M} \binom{m}{M-(n-j)}\binom{n}{n-j}a^{M-(n-j)} b^{n-j}, \quad \text{as we need  $M-(n-j) \leq m$}
\end{split}
\end{equation*}
\begin{equation*}\label{eqn thm triple1c2}
\begin{split}
&\equiv \sum_{j=M-n}^{m} \binom{m}{j}\binom{n}{M-j}a^{j} b^{M-j}, \quad \text{by the change of index.}
\end{split}
\end{equation*}
This completes the proof of the theorem.
\end{proof}
Next, we have a direct consequence of \thmref{thm triple power1}.
\begin{corollary}\label{cor triple power1}
Let $p$ be a prime number. For any integers $a, \, b, \, m, \, n, \, s \in \{1, \, 2, \, \ldots, \, p-1\}$ with $a \neq b$, we have the following congruences modulo $p$:
\begin{equation*}\label{eqn cor triple1}
\begin{split}
\sum_{k=0}^{p-1} (a+k)^m (b+k)^n k^s \equiv 
\begin{cases}
0,  &\quad \text{if $m+n+s < p-1$}\\
-1,  &\quad \text{if $m+n+s = p-1$}\\
-(ma+nb),  &\quad \text{if $m+n+s = p$}\\
- \binom{m}{2}a^2-mnab- \binom{n}{2}b^2,  &\quad \text{if $m+n+s = p+1$}\\
-3,   &\quad \text{if $m=n=s=p-1$},
\end{cases}
\end{split}
\end{equation*}
\end{corollary}
\begin{proof}
Use \thmref{thm triple power1} with $M$ is negative, $0$, $1$ or $2$ in the first four cases. 
\end{proof}
We want to compute $\sum_{k=0}^{p} (a+k)^m (b+k)^{n}k^s$ in another way. We need a preparatory lemma:
\begin{lemma}\label{lem triple power1a}
For any positive integers $p, \, a, \, b, \, m, \, n, \, s$, we have the following equalities:
\begin{equation*}\label{eqn lem triple1a}
\begin{split}
\sum_{k=0}^{p} (a+k)^m (b+k)^{n+s} &= \sum_{i=0}^{s} \binom{s}{i}b^{s-i} \sum_{k=0}^{p} (a+k)^m (b+k)^n k^i,\\
\sum_{k=0}^{p} (a+k)^m (b+k)^{n}k^s &=\sum_{k=0}^{p} (a+k)^m (b+k)^{n+1}k^{s-1}-b\sum_{k=0}^{p} (a+k)^m (b+k)^{n}k^{s-1}\\
&=\sum_{k=0}^{p} (a+k)^{m+1} (b+k)^{n}k^{s-1}-a\sum_{k=0}^{p} (a+k)^m (b+k)^{n}k^{s-1}.
\end{split}
\end{equation*}
\end{lemma}
\begin{proof}
Since
$\sum_{k=0}^{p} (a+k)^m (b+k)^{n+s}=\sum_{k=0}^{p} (a+k)^m (b+k)^{n}(b+k)^{s}$,
the first equality follows by applying Binomial Theorem to $(b+k)^{s}$.
The second equality is derived as below:
\begin{equation*}\label{eqn lem triple1aa}
\begin{split}
\sum_{k=0}^{p} (a+k)^m (b+k)^{n+1}k^{s-1}&=\sum_{k=0}^{p} (a+k)^m (b+k)^{n}k^{s-1}(b+k)\\
&=b\sum_{k=0}^{p} (a+k)^m (b+k)^{n}k^{s-1}+\sum_{k=0}^{p} (a+k)^m (b+k)^{n}k^{s}.
\end{split}
\end{equation*}
The third equality is similar to the second one.
\end{proof}
We can use the equalities in \lemref{lem triple power1a} to compute  $\sum_{k=0}^{p} (a+k)^m (b+k)^{n}k^s$ successively.
For example, when $s=1$ we have the following result.

\begin{theorem}\label{thm triple power2}
Let $p$ be a prime number. For any integers $a, \, b, \, m, \, n, \, s \in \{1, \, 2, \, \ldots, \, p-1\}$ with $a \neq b$, we have the following congruence modulo $p$:
\begin{equation*}\label{eqn thm triple2}
\begin{split}
\sum_{k=0}^{p-1} (a+k)^m (b+k)^n k \equiv  
\begin{cases} 
(a+b),  \quad \text{if $m=n=p-1$}\\
-2-b(a-b)^{p-2},  \quad \text{if $m=p-1$ and $n=p-2$}\\
-2+a(a-b)^{p-2},  \quad \text{if $m=p-2$ and $n=p-1$}\\
-(a-b)^{m+n+1}\binom{m}{p-n-2} +b (a-b)^{m+n}\binom{m}{p-n-1}, \, \, \, \text{otherwise}.
\end{cases}
\end{split}
\end{equation*}
\end{theorem}
\begin{proof}
If $m=n=p-1$, $\displaystyle{\sum_{k=0}^{p-1} (a+k)^m (b+k)^n k \equiv \sum_{\substack{k=1 \\ k \not \equiv -a, \ -b}}^{p-1}k \equiv (a+b)}$ by \eqref{eqn FLT} and \thmref{thm power sum0}.
%
%
This proves the first congruence.
By \lemref{lem triple power1a}, 
\begin{equation*}\label{eqn thm triple20}
\begin{split}
\sum_{k=0}^{p-1} (a+k)^m (b+k)^n k &=\sum_{k=0}^{p-1} (a+k)^m (b+k)^{n+1}-b\sum_{k=0}^{p-1} (a+k)^m (b+k)^{n}\\
&=\sum_{k=0}^{p-1} (a+k)^{m+1} (b+k)^{n}-a\sum_{k=0}^{p-1} (a+k)^m (b+k)^{n}.
\end{split}
\end{equation*}
Then the result follows from \thmref{thmcor new power sum3 second}.
\end{proof}
\begin{theorem}\label{thm triple power2s2}
Let $p$ be a prime number. For any integers $a, \, b, \, m, \, n, \, s \in \{1, \, 2, \, \ldots, \, p-1\}$ with $a \neq b$, we have the following congruence modulo $p$:
\begin{equation*}\label{eqn thm triple2s2}
\begin{split}
\sum_{k=0}^{p-1} (a+k)^m (b+k)^n k^2 \equiv  
\begin{cases} 
-(a^2+b^2),  \quad \text{if $m=n=p-1$}\\
2b+a^2(a-b)^{p-2},  \quad \text{if $m=p-1$ and $n=p-2$}\\
2a-b^2(a-b)^{p-2},  \quad \text{if $m=p-2$ and $n=p-1$}\\
-1+2ab(a-b)^{p-3},  \quad \text{if $m=n=p-2$}\\
J, \quad \text{otherwise},
\end{cases}
\end{split}
\end{equation*}
where 
$$J:=-(a-b)^{m+n+2} \binom{m}{p-n-3}+2b(a-b)^{m+n+1}\binom{m}{p-n-2}-b^2(a-b)^{m+n}\binom{m}{p-n-1}.$$
\end{theorem}
\begin{proof}
If $m=n=p-1$, $\displaystyle{\sum_{k=0}^{p-1} (a+k)^m (b+k)^n k^2 \equiv \sum_{\substack{k=1 \\ k \not \equiv -a, \ -b}}^{p-1}k^2 \equiv -(a^2+b^2)}$ by \eqref{eqn FLT} and \thmref{thm power sum0}.
This proves the first congruence.

If $m=p-1$ and $n=p-2$, we obtain the the second congruence as follows:
\begin{equation*}\label{eqn thm triple2s2a}
\begin{split}
\sum_{k=0}^{p-1} (a+k)^m (b+k)^n k^2 &\equiv \sum_{\substack{k=1 \\ k \not \equiv -a}}^{p-1}(b+k)^{p-2} k^2
\equiv \sum_{k=1}^{p-1}(b+k)^{p-2} k^2 -(b-a)^{p-2}a^2, \quad \text{by \eqref{eqn FLT}}\\
&\equiv -b\binom{p-2}{p-3} -(b-a)^{p-2}a^2, \quad \text{by \thmref{thmcor new power sum3}}\\
& \equiv 2b+a^2(a-b)^{p-2}.
\end{split}
\end{equation*}
If $m=p-2$ and $n=p-1$, by \eqref{eqn FLT} 

\begin{equation*}\label{eqn thm triple2s2aa}
\begin{split}
\sum_{k=0}^{p-1} (a+k)^m (b+k)^n k^2 &\equiv \sum_{\substack{k=1 \\ k \not \equiv -b}}^{p-1}(a+k)^{p-2} k^2
\equiv \sum_{k=1}^{p-1}(a+k)^{p-2} k^2 -(a-b)^{p-2}b^2\\
&\equiv -a\binom{p-2}{p-3} -(a-b)^{p-2}b^2, \quad \text{by \thmref{thmcor new power sum3}}\\
&\equiv 2a-b^2(a-b)^{p-2}.
\end{split}
\end{equation*}
This proves the third congruence.
By \lemref{lem triple power1a}, 
\begin{equation}\label{eqn thm triple20s2b}
\begin{split}
\sum_{k=0}^{p-1} (a+k)^m (b+k)^n k^2 =\sum_{k=0}^{p-1} (a+k)^m (b+k)^{n+1}k-b\sum_{k=0}^{p-1} (a+k)^m (b+k)^{n}k.
\end{split}
\end{equation}
If $m=p-2$ and $n=p-2$, by \eqref{eqn thm triple20s2b} and \thmref{thm triple power2} we have
\begin{equation*}\label{eqn thm triple2s2bb}
\begin{split}
&\sum_{k=0}^{p-1} (a+k)^m (b+k)^n k^2 
\equiv -2+a(a-b)^{p-2} -b \Big[ -(a-b)^{p-2+p-1} 
+b(a-b)^{p-2+p-2}(p-2)
\Big] \\
&\equiv -2+a(a-b)^{p-2}+b(a-b)^{p-2}+2b^2(a-b)^{p-3}, \quad \text{by \eqref{eqn FLT}}\\
&\equiv -1+2b(a-b)^{p-2}+2b^2(a-b)^{p-3}, \quad \text{since $1 \equiv a(a-b)^{p-2}-b(a-b)^{p-2}$}\\
&\equiv -1+2ab(a-b)^{p-3}.
\end{split}
\end{equation*}
This proves the fourth congruence.
For the remaining values of $m$ and $n$, we can use \eqref{eqn thm triple20s2b} and apply the last congruence in \thmref{thm triple power2} to derive
\begin{equation*}\label{eqn thm triple2s2c}
\begin{split}
&\sum_{k=0}^{p-1} (a+k)^m (b+k)^n k^2 \equiv  -(a-b)^{m+n+2} \binom{m}{p-n-3}+b(a-b)^{m+n+1}\binom{m}{p-n-2}\\
& \qquad \qquad \qquad \quad \quad \quad \quad \quad \quad 
-b \Big[ -(a-b)^{m+n+1} \binom{m}{p-n-2}+b(a-b)^{m+n}\binom{m}{p-n-1}\Big] \\
&\equiv -(a-b)^{m+n+2} \binom{m}{p-n-3}+2b(a-b)^{m+n+1}\binom{m}{p-n-2}-b^2(a-b)^{m+n}\binom{m}{p-n-1}.
\end{split}
\end{equation*}
This completes the proof of the theorem.
\end{proof}
More generally, we have
\begin{theorem}\label{thm triple power2sg}
Let $p$ be a prime number, and let $a, \, b, \, m, \, n, \, s \in \{1, \, 2, \, \ldots, \, p-1\}$ with $a \neq b$.
We have the following congruence modulo $p$:
\begin{equation*}\label{eqn thm triple2sg}
\begin{split}
\sum_{k=0}^{p-1} &(a+k)^m (b+k)^n k^s \equiv \\ 
&\begin{cases} 
-3,  \quad \text{if $m=n=s=p-1$}\\
-b^{n+s-(p-1)} \binom{n}{p-1-s} - (b-a)^{n}(-a)^s,  \quad \text{if $m=p-1$ and $(n,s) \neq (p-1,p-1)$}\\
-\sum_{j=0}^{M-m-1}  \binom{m}{R-j} \binom{n}{j}a^{R-j} b^j -\sum_{j=M-m}^{n}  \binom{m}{M-j} \binom{n}{j}a^{M-j} b^j,  \quad \text{otherwise},
\end{cases}
\end{split}
\end{equation*}
where $M=m+n+s-(p-1)$ and $R=m+n+s-2(p-1)$.
\end{theorem}
\begin{proof}
When $m=p-1$, 
\begin{equation*}\label{eqn thm triple2sgg}
\begin{split}
\sum_{k=0}^{p-1} (a+k)^m (b+k)^n k^s \equiv \sum_{\substack{k=0 \\ k \not \equiv -a}}^{p-1} (b+k)^n k^s
\equiv \sum_{k=0}^{p-1} (b+k)^n k^s-(b-a)^n(-a)^s.
\end{split}
\end{equation*}
Thus, the first and the second congruences in the theorem follows from \thmref{thmcor new power sum3}.
Next, we consider the case $m<p-1$.
\begin{equation*}\label{eqn thm triple2sgh}
\begin{split}
\sum_{k=0}^{p-1} (a+k)^m (b+k)^n k^s &=  \sum_{k=0}^{p-1} (a+k)^m k^s \sum_{j=0}^{n} \binom{n}{j} b^j k^{n-j},
\quad \text{ by Binomial Theorem}\\ 
&=\sum_{j=0}^{n} \binom{n}{j} b^j  \sum_{k=0}^{p-1} (a+k)^m k^{s+n-j}.
\end{split}
\end{equation*}
When $s+n-j > p-1$,  $k^{s+n-j} \equiv k^{s+n-j-(p-1)}$. Thus, we continue
\begin{equation*}\label{eqn thm triple2sgk}
\begin{split} 
& \equiv \sum_{j=0}^{s+n-(p-1)-1} \binom{n}{j} b^j  \sum_{k=0}^{p-1} (a+k)^m k^{s+n-j-(p-1)} + \sum_{j=s+n-(p-1)}^{n} \binom{n}{j} b^j  \sum_{k=0}^{p-1} (a+k)^m k^{s+n-j}\\
&\equiv -\sum_{j=0}^{s+n-(p-1)-1}  \binom{m}{R-j} \binom{n}{j}a^{R-j} b^j -\sum_{j=s+n-(p-1)}^{n}  \binom{m}{M-j} \binom{n}{j}a^{M-j} b^j\\
&\equiv -\sum_{j=0}^{M-m-1}  \binom{m}{R-j} \binom{n}{j}a^{R-j} b^j -\sum_{j=M-m}^{n}  \binom{m}{M-j} \binom{n}{j}a^{M-j} b^j,
\end{split}
\end{equation*}
where the second congruence follows from \thmref{thmcor new power sum3} as $m<p-1$.
\end{proof}
Note that when $s=1$ and $m+n+2-p \geq 0$, \thmref{thm triple power1} and \thmref{thm triple power2} give formulas for 
$\sum_{k=0}^{p-1} (a+k)^m (b+k)^n k$ in two different forms. We obtain the following congruence by comparing those two formulas:
\begin{theorem}\label{thm triple comp1}
Let $p$ be a prime number, and let $a, \, b, \, m, \, n \in \{1, \, 2, \, \ldots, \, p-1\}$ with $a \neq b$. Suppose $M=m+n+2-p$ and $M < p-1$ we have the following congruence modulo $p$:
\begin{equation*}\label{eqn thm triple comp1}
\begin{split}
\sum_{j=0}^{M}\binom{m}{M-j}\binom{n}{j}a^{M-j}b^j \equiv (a-b)^{m+n+1} \binom{m}{p-n-2} -b (a-b)^{m+n} \binom{m}{p-n-1}.
\end{split}
\end{equation*}
\end{theorem}
If we apply Binomial Theorem to the terms $(a-b)^{m+n+1}$ and $(a-b)^{m+n}$ in \thmref{thm triple comp1}, we notice that not only the congruence in \thmref{thm triple comp1} holds, but also the following congruences of the coefficients of each terms in the summations hold:
\begin{theorem}\label{thm triple cong1}
Let $p$ be a prime number, and let $m, \, n \in \{1, \, 2, \, \ldots, \, p-1\}$. Suppose $M=m+n+2-p$ and $M < p-1$. For every $0 \leq j \leq M$, we have the following congruence modulo $p$:
\begin{equation*}\label{eqn thm triple cong1}
\begin{split}
(-1)^j\binom{m}{M-j}\binom{n}{j} \equiv \binom{m}{M}\binom{M}{j} +\binom{m}{M-1}\binom{M-1}{j-1} .
\end{split}
\end{equation*}
\end{theorem}
\begin{proof}
By the cancellation identity,
\begin{equation*}\label{eqn thm triple cong6}
\begin{split}
&\binom{m}{M}\binom{M}{j} +\binom{m}{M-1}\binom{M-1}{j-1}  = \binom{m}{j}\binom{m-j}{M-j} +\binom{m}{j-1}\binom{m-j+1}{M-j}\\
&=\binom{m}{m-j}\binom{m-j}{M-j} +\binom{m}{m-j+1}\binom{m-j+1}{M-j}, \quad \text{by the symmetry property}\\
&=\binom{m}{M-j}\binom{m-M+j}{m-M} +\binom{m}{M-j}\binom{m-M+j}{m-M+1}, \quad \text{by the cancellation identity}\\
&=\binom{m}{M-j}\binom{m-M+j+1}{m-M+1}, \quad \text{by the Pascal's identity}\\
&=\binom{m}{M-j}\binom{m-M+j+1}{j}, \quad \text{by the symmetry property}\\
&=\binom{m}{M-j}\binom{p-1-n+j}{j}, \quad \text{by the definition of $M$}\\
&\equiv (-1)^j\binom{m}{M-j}\binom{n}{j}, \quad 
\text{by the 2nd congruence in \eqref{eqn thms CO}}.
\end{split}
\end{equation*}
This completes the proof.
\end{proof}
When $s=2$ and $m+n+3-p \geq 0$, both \thmref{thm triple power1} and \thmref{thm triple power2s2} express 
$\sum_{k=0}^{p-1} (a+k)^m (b+k)^n k^2$ in two different ways. This leads to the following congruence:
\begin{theorem}\label{thm triple comp2}
Let $p$ be a prime number, and let $a, \, b, \, m, \, n \in \{1, \, 2, \, \ldots, \, p-1\}$ with $a \neq b$. Suppose $M=m+n+3-p$ and $M< p-1$ we have the following congruence modulo $p$:
\begin{equation*}\label{eqn thm triple comp2}
\begin{split}
\sum_{j=0}^{M}\binom{m}{M-j}\binom{n}{j}a^{M-j}b^j &\equiv (a-b)^{m+n+2} \binom{m}{p-n-3} -2b (a-b)^{m+n+1} \binom{m}{p-n-2}\\
&\quad \quad +b^2 (a-b)^{m+n} \binom{m}{p-n-1}.
\end{split}
\end{equation*}
\end{theorem}
If we apply Binomial Theorem to the terms $(a-b)^{m+n+2}$, $(a-b)^{m+n+1}$ and $(a-b)^{m+n}$ in \thmref{thm triple comp2}, we notice that the following congruences of the coefficients of each terms in those summations hold:

\begin{theorem}\label{thm triple cong2}
Let $p$ be a prime number, and let $m, \, n \in \{1, \, 2, \, \ldots, \, p-1\}$. Suppose $M=m+n+3-p$ and $M < p-1$. For every $0 \leq j \leq M$, we have the following congruence modulo $p$:
\begin{equation*}\label{eqn thm triple cong2a}
\begin{split}
(-1)^j\binom{m}{M-j}\binom{n}{j} \equiv \binom{m}{M}\binom{M}{j} +2\binom{m}{M-1}\binom{M-1}{j-1} +\binom{m}{M-2}\binom{M-2}{j-2} .
\end{split}
\end{equation*}
\end{theorem}
The proof of \thmref{thm triple cong2} can be obtained by following similar arguments used in the proof of \thmref{thm triple cong1}. Moreover, we noticed by doing computations in \cite{MMA} that \thmref{thm triple cong1} and \thmref{thm triple cong2} can be generalized as follows:
\begin{theorem}\label{thm triple cong3}
Let $p$ be a prime number, and let $m, \, n, \, s \in \{0, \, 1, \, \ldots, \, p-1\}$. Suppose $M=m+n+s-(p-1)$ and $0 \leq M < p-1$. For every $0 \leq j \leq M$, we have
\begin{equation*}\label{eqn thm triple cong3a}
\begin{split}
(-1)^j\binom{m}{M-j}\binom{n}{j} \equiv \sum_{k=0}^{s}\binom{s}{k}\binom{m}{M-k}\binom{M-k}{j-k}  \, \, \, mod \, \, p
\end{split}
\end{equation*}
\end{theorem}
\begin{proof}
Let $m$, $n$, $s$, $M$, $j$ and $p$ be as in the theorem. By the cancellation identity,
\begin{equation*}
\begin{split}
&\sum_{k=0}^{s}\binom{s}{k}\binom{m}{M-k}\binom{M-k}{j-k} 
 = \sum_{k=0}^{s}\binom{s}{k}\binom{m}{j-k}\binom{m-j+k}{M-j},\\
& = \sum_{k=0}^{s}\binom{s}{k}\binom{m}{m-j+k}\binom{m-j+k}{M-j},  \quad \text{by the symmetry property}\\
& = \sum_{k=0}^{s}\binom{s}{k}\binom{m}{M-j}\binom{m-M+j}{m+k-M},  \quad \text{by the cancellation identity}\\
& = \binom{m}{M-j} \sum_{k=0}^{s}\binom{s}{k}\binom{m-M+j}{m+k-M},  \\
& = \binom{m}{M-j} \sum_{k=0}^{s}\binom{s}{k}\binom{m-M+j}{j-k},  \quad \text{by the symmetry property}\\
& = \binom{m}{M-j} \sum_{k=0}^{j}\binom{s}{k}\binom{m-M+j}{j-k}, \quad \text{as $\binom{s}{k}=0$ if $k>s$
and we need $j \geq k$}\\
& = \binom{m}{M-j} \binom{s+m-M+j}{j}, \quad \text{by Vandermonde's identity}\\
& = \binom{m}{M-j} \binom{p-1-n+j}{j}, \quad \text{by the definition of $M$}\\
& \equiv (-1)^j\binom{m}{M-j} \binom{n}{j} \, \, \, mod \, \, p, \quad 
\text{by the 2nd congruence in \eqref{eqn thms CO}}.
\end{split}
\end{equation*}
This completes the proof.
\end{proof}

\begin{corollary}\label{cor triple comp0}
Let $p$ be a prime number, and let $a, \, b, \, m, \, n \in \{1, \, 2, \, \ldots, \, p-1\}$ with $a \neq b$. Let $M=m+n-(p-1)$. For every $0 \leq j \leq M$, we have the following congruences:
\begin{equation*}\label{eqn cor triple cong0a}
\begin{split}
(-1)^j\binom{m}{M-j}\binom{n}{j} \equiv \binom{m}{M}\binom{M}{j} \, \, \, mod \, \, p
\end{split}
\end{equation*}
Moreover,
\begin{equation*}\label{eqn cor triple comp0}
\begin{split}
\sum_{j=0}^{M}\binom{m}{M-j}\binom{n}{j}a^{M-j}b^j \equiv (a-b)^{M} \binom{m}{p-n-1}  \, \, \, mod \, \, p
\end{split}
\end{equation*}
\end{corollary}
\begin{proof}
The first congruence follows directly when $m=n=p-1$. Otherwise, it follows from \thmref{thm triple cong2} by taking $s=0$. We use the first congruence to obtain

\begin{equation*}\label{eqn cor triple comp00}
\begin{split}
&\sum_{j=0}^{M}\binom{m}{M-j}\binom{n}{j}a^{M-j}b^j \equiv 
\sum_{j=0}^{M}(-1)^j\binom{m}{M}\binom{M}{j}a^{M-j}b^j
\equiv \binom{m}{M} \sum_{j=0}^{M}\binom{M}{j}a^{M-j}(-b)^j\\
& =  \binom{m}{M} (a-b)^{M}, \quad \text{by Binomial Theorem}\\
& = \binom{m}{p-n-1} (a-b)^{m+n-(p-1)}, \quad \text{by the definition of $M$ and the symmetry property}.
\end{split}
\end{equation*}
This completes the proof.
\end{proof}
Having derived \thmref{thm triple cong3}, we can generalize the second part of \corref{cor triple comp0}, \thmref{thm triple comp1} and \thmref{thm triple comp2} as follows: 
\begin{theorem}\label{thm triple comp3}
Let $p$ be a prime number, and let $a, \, b, \, m, \, n, \, s \in \{1, \, 2, \, \ldots, \, p-1\}$ with $a \neq b$.
Let $M=m+n+s-(p-1) < p-1$. We have the following congruence modulo $p$:
\begin{equation*}\label{eqn thm triple comp3}
\begin{split}
\sum_{j=0}^{M}\binom{m}{M-j}\binom{n}{j}a^{M-j}b^j 
\equiv \sum_{k=0}^{s}\binom{m}{M-k}\binom{s}{k}(a-b)^{M-k}(-b)^k.
\end{split}
\end{equation*}
\end{theorem}
\begin{proof}
By using \thmref{thm triple cong3},
\begin{equation*}\label{eqn thm triple comp3a}
\begin{split}
&\sum_{j=0}^{M}\binom{m}{M-j}\binom{n}{j}a^{M-j}b^j 
\equiv \sum_{j=0}^{M}\sum_{k=0}^{s} (-1)^j \binom{s}{k}\binom{m}{M-k}\binom{M-k}{j-k}  a^{M-j}b^j\\
& = \sum_{k=0}^{s}  \binom{s}{k}\binom{m}{M-k} \sum_{j=0}^{M} \binom{M-k}{j-k}  a^{M-j}(-b)^j\\
& = \sum_{k=0}^{s}  \binom{s}{k}\binom{m}{M-k} \sum_{j=k}^{M} \binom{M-k}{j-k}  a^{M-j}(-b)^j\\
& = \sum_{k=0}^{s}  \binom{s}{k}\binom{m}{M-k} \sum_{j=0}^{M-k} \binom{M-k}{j}  a^{M-k-j}(-b)^{j+k}\\
& = \sum_{k=0}^{s}  \binom{s}{k}\binom{m}{M-k} (a-b)^{M-k} (-b)^{k},
\end{split}
\end{equation*}
where the last equality follows from Binomial Theorem.
\end{proof}


\begin{remark}\label{rem sums}
Given $a, \, b, \, c, \, m, \, n, \, s \in \{1, \, 2, \, \ldots, \, p-1\}$, we have
$$\sum_{k=0}^{p-1} (a+k)^m (b+k)^n (c+k)^s \equiv  \sum_{k=0}^{p-1} (a-c+k)^m (b-c+k)^n k^s \, \, \, mod \, \, p.$$
Moreover, for each $d \in \{1, \, 2, \, \ldots, \, p-1\}$ we have
$$\sum_{k=0}^{p-1} (a+k)^m (b+k)^n (c+k)^s \equiv \sum_{k=0}^{p-1} (a-d+k)^m (b-d+k)^n (c-d+k)^s \, \, \, mod \, \, p.$$
\end{remark}

\begin{remark}\label{rem triple all}
Let $p$ be a prime number, and let  $m, \, n, \, s \in \{1, \, 2, \, \ldots, \, p-1\}$. For any distinct integers $a, \, b, \, c \in \{0, \, 1, \, \ldots, \, p-1\}$, we have the following congruences modulo $p$:
\begin{equation*}\label{eqn rem triple all}
\begin{split}
&\sum_{\substack{k=0 \\ k \not \equiv -c}}^{p-1} \frac{(a+k)^m (b+k)^n} {(c+k)^s} \equiv \sum_{k=0}^{p-1} (a+k)^m (b+k)^n (c+k)^{p-1-s}, \quad \text{if $s \neq p-1$}.\\ 
&\sum_{\substack{k=0 \\ k \not \equiv -b \\ k \not \equiv -c}}^{p-1} \frac{(a+k)^m } {(b+k)^n (c+k)^s} \equiv \sum_{k=0}^{p-1} (a+k)^m (b+k)^{p-1-n} (c+k)^{p-1-s}, \quad \text{if $n \neq p-1$ and $ s \neq p-1$}.
\end{split}
\end{equation*}
Moreover, if $m \neq p-1$, $n \neq p-1$ and $s \neq p-1$, we have
$$\sum_{\substack{k=0 \\ k \not \equiv -a \\ k \not \equiv -b, \, \, k \not \equiv -c}}^{p-1} \frac{1} {(a+k)^m (b+k)^n (c+k)^s } \equiv \sum_{k=0}^{p-1} (a+k)^{p-1-m} (b+k)^{p-1-n} (c+k)^{p-1-s}.$$
\end{remark}
Note that each of these three summations can be computed by using \remref{rem sums}, \thmref{thm triple power1},
\thmref{thm triple power2}, \thmref{thm triple power2s2} and \thmref{thm triple power2sg}.
For example,
\begin{corollary}\label{cor triple all}
Let $p$ be a prime number, and let  $m, \, n, \, s \in \{1, \, 2, \, \ldots, \, p-1\}$. For any distinct integers $a, \, b, \, c \in \{0, \, 1, \, \ldots, \, p-1\}$, we have the following congruences modulo $p$:
\begin{equation*}\label{eqn cor triple all}
\begin{split}
&\sum_{\substack{k=0 \\ k \not \equiv -c}}^{p-1} \frac{(a+k)^m (b+k)^n} {(c+k)^s} \equiv 0, \quad \text{if $m+n < s \neq p-1$}.\\ 
&\sum_{\substack{k=0 \\ k \not \equiv -b, \, \, k \not \equiv -c}}^{p-1} \frac{(a+k)^m } {(b+k)^n (c+k)^s} \equiv 0, \quad \text{if $p-1<n+s-m$, $n \neq p-1$ and $ s \neq p-1$}.
\end{split}
\end{equation*}
Moreover, if $2(p-1)<m+n+s$, $m \neq p-1$, $n \neq p-1$ and $s \neq p-1$, we have
$$\sum_{\substack{k=0 \\ k \not \equiv -a \\ k \not \equiv -b, \, \, k \not \equiv -c}}^{p-1} \frac{1} {(a+k)^m (b+k)^n (c+k)^s } \equiv 0.$$
\end{corollary}

\textbf{Example:} Let $p=17$. We want to compute $\displaystyle{I=\sum_{\substack{k=0 \\ k \not \equiv -3, \, \, k \not \equiv -8}}^{p-1} \frac{(7+k)^9 } {(3+k)^{13} (8+k)^8}}$. 
We obtain the following congruences modulo $p$ by using \remref{rem sums} and \remref{rem triple all}:
$$
I \equiv  \sum_{k=0}^{p-1} (3+k)^3(8+k)^8 (7+k)^9 \equiv \sum_{k=0}^{p-1} (13+k)^3(1+k)^8 k^9.
$$
Then, we use \thmref{thm triple power2sg} with $M=3+8+9-(17-1)=4$ to derive 
$$
I \equiv -\sum_{j=4-3}^8 \binom{3}{4-j} \binom{8}{j} 13^{4-j} 1^j\\
\equiv -\sum_{j=1}^4 \binom{3}{4-j} \binom{8}{j} (-4)^{4-j} \equiv 8.\\
$$
\textbf{Example:} Let $p=23$. By \corref{cor triple all} we have 
$$\displaystyle{\sum_{\substack{k=0 \\ k \not \equiv -7, \, \, k \not \equiv -13, \, \, k \not \equiv -18}}^{p-1} \frac{1} {(7+k)^{16} (13+k)^{17} (18+k)^{19}}} \equiv 0.$$
%
The following result follows from \thmref{thm triple power2sg}:
\begin{corollary}\label{cor triple constant}
Let $p$ be a prime number, and let  $m, \, n, \, s \in \{1, \, 2, \, \ldots, \, p-1\}$. For any distinct integers $a, \, b \in 
\{1, \, \ldots, \, p-1\}$, we have the following congruence modulo $p$:
\begin{equation*}\label{eqn cor triple constant}
\begin{split}
\sum_{k=0}^{p-1} (a+k)^m (b+k)^n k^s \equiv b^{M} \sum_{k=0}^{p-1} (ab^{-1}+k)^m (1+k)^n k^s, 
\end{split}
\end{equation*}
where $M=m+n+s-(p-1)$.
\end{corollary}

\section{The sums of ratios or products in general}\label{sec general}
In this section, we consider the sums of the products or ratios of the terms $(a_1+k)^{m_1}$, ..., $(a_n+k)^{m_n}$.
All congruences in this section are considered modulo a prime number $p$. At the end, we express such sums in terms of the certain coefficients of the polynomial $(b_1+x)^{m_1} \cdots (b_{n-1}+x)^{m_{n-1}}$. That is, we relate such sums to Elementary Symmetric Functions. Computational experiments done in \cite{MMA} enabled us to establish such relations, which were not easy to see otherwise.

Note that we can extend the proof of \thmref{thm triple power1} to derive the following generalized result which should be considered along with the extension of \remref{rem triple all}:
\begin{theorem}\label{thm general power1}
Given a prime number $p$, let $m_{i}$ and $a_{i}$ be integers with $1 \leq m_i \leq p-1$ and 
$0 \leq a_i \leq p-1$ for each $i=1, \, 2, \, \ldots, \, n $. If $a_i \neq a_j$ when $i \neq j$, we have the following congruences modulo $p$:
\begin{equation*}
\begin{split}
\sum_{k=0}^{p-1} (a_1+k)^{m_1}  \ldots  (a_n+k)^{m_n} \equiv
&\begin{cases} 
0,  \quad \text{if $m_1+\cdots+m_n < p-1$}\\
-n,  \quad \text{if $m_1=m_2= \cdots =m_n=p-1$}\\
J,  \quad \text{otherwise},
\end{cases}
\end{split}
\end{equation*}
where $J$ is the summation given below with $b_{i'}=a_{i'}-a_n$ for each index $i'=1, \, 2, \, \ldots, \, n-1$, $M_i=m_1+\cdots+m_n-i(p-1)$ and $t$ is the integer such that $M_{t+1}<0 \leq M_t$.
\begin{equation*}
\begin{split}
J=-\sum_{\substack{ i=1}}^t \, \,
\sum_{\substack{j_1+\cdots+j_{n-1}=M_i \\ 0 \leq j_1 \leq m_1, \, \ldots, \, 0 \leq j_{n-1} \leq m_{n-1}}}
\binom{m_1}{j_1}  \cdots \binom{m_{n-1}}{j_{n-1}} b_1^{j_1} \cdots b_{n-1}^{j_{n-1}}.
\end{split}
\end{equation*} 
\end{theorem}
\begin{proof}
Let $\displaystyle{I:=\sum_{k=0}^{p-1} (a_1+k)^{m_1}  \ldots  (a_n+k)^{m_n}}$.

When $m_1=m_2= \cdots =m_n=p-1$, we obtain the second congruence in the theorem by \eqref{eqn FLT} as follows:
$$
I
\equiv \sum_{\substack{0 \leq k \leq p-1 \\ k \not \equiv a_1, \, \cdots, \, k \not \equiv a_n }} 1  \equiv -n.
$$
Suppose at least one of  $m_i$'s is less than $p-1$. We first note that the generalized version of \remref{rem sums} holds and so we have the following first congruence, where we set $b_{i'}=a_{i'}-a_n$ for each index $i'=1, \, 2, \, \ldots, \, n-1$:
\begin{equation}\label{eqn general power1a}
\begin{split}
&I
\equiv \sum_{k=0}^{p-1} (b_1+k)^{m_1}  \cdots  (b_{n-1}+k)^{m_{n-1}}k^{m_n},\\
&\equiv \sum_{k=0}^{p-1} k^{m_n} \sum_{j_1=0}^{m_1} \binom{m_1}{j_1} b_1^{j_1}k^{m_1-j_1} \ldots \sum_{j_{n-1}=0}^{m_{n-1}} \binom{m_{n-1}}{j_{n-1}} b_{n-1}^{j_{n-1}}k^{m_{n-1}-j_{n-1}}, \quad \text{by Binomial Theorem}\\
&\equiv \sum_{j_1=0}^{m_1}  \cdots   \sum_{j_{n-1}=0}^{m_{n-1}} \binom{m_1}{j_1}  \cdots \binom{m_{n-1}}{j_{n-1}} b_1^{j_1} \cdots b_{n-1}^{j_{n-1}} \sum_{k=0}^{p-1} k^{m_1-j_1+m_2-j_2+\cdots+m_{n-1}-j_{n-1}+m_n}.\\
\end{split}
\end{equation}
Note that $1 \leq m_n \leq m_1-j_1+m_2-j_2+\cdots+m_{n-1}-j_{n-1}+m_n \leq m_1+\cdots +m_n$. Thus, when 
$m_1+\cdots +m_n < p-1$, $I \equiv 0$ by  \thmref{thm power sum0}. This proves the first congruence in the theorem. 

When $m_1+\cdots +m_n \geq p-1$, we set $M_i=m_1+\cdots +m_n -i(p-1)$ for each integer $i$ providing that $M_i \geq 0$. Let $t$ be  the largest of such $i$, i.e., $M_{t+1} < 0 \leq M_t$. Then we apply \thmref{thm power sum0} in \eqref{eqn general power1a} to derive
\begin{equation*}\label{eqn general power1b}
\begin{split}
I&\equiv -\sum_{\substack{ i=1}}^t \, \,
\sum_{\substack{j_1+\cdots+j_{n-1}=M_i \\ 0 \leq j_1 \leq m_1, \, \ldots, \, 0 \leq j_{n-1} \leq m_{n-1}}}
\binom{m_1}{j_1}  \cdots \binom{m_{n-1}}{j_{n-1}} b_1^{j_1} \cdots b_{n-1}^{j_{n-1}}.
\end{split}
\end{equation*}
This completes the proof of the theorem.
\end{proof}
The following result is about the simple cases that are easy consequences of \thmref{thm general power1}:
\begin{corollary}\label{cor general power1}
Given a prime number $p$, let $m_{i}$ and $a_{i}$ be integers with $1 \leq m_i \leq p-1$ and 
$0 \leq a_i \leq p-1$ for each $i=1, \, 2, \, \ldots, \, n $. If $a_i \neq a_j$ when $i \neq j$, we have the following congruences modulo $p$:
\begin{equation*}\label{eqn cor general power1}
\begin{split}
\sum_{k=0}^{p-1} (a_1+k)^{m_1} \cdots (a_n+k)^{m_n} \equiv 
\begin{cases}
0,  &\quad \text{if $m_1+m_2+ \cdots + m_n < p-1$}\\
-1,  &\quad \text{if $m_1+m_2+ \cdots + m_n = p-1$}\\
-\sum_{j=1}^{n}m_j a_j, &\quad \text{if $m_1+m_2+ \cdots + m_n = p$}\\
-n,   &\quad \text{if $m_1=m_2= \cdots = m_n=p-1$},
\end{cases}
\end{split}
\end{equation*}
\end{corollary}
The following result is a generalization of \corref{cor triple constant}:
\begin{corollary}\label{cor general constant}
Given a prime number $p$, let $m_{i}$ and $b_{i}$ be integers with $1 \leq m_i \leq p-1$ and 
$1 \leq b_i \leq p-1$ for each $i=1, \, 2, \, \ldots, \, n-1 $. If $b_i \neq b_j$ when $i \neq j$, we have the following congruence modulo $p$:
\begin{equation*}\label{eqn cor general constant}
\begin{split}
\sum_{k=0}^{p-1} (b_1+k)^{m_1}  \cdots (b_{n-1}+k)^{m_{n-1}}k^{m_n} \equiv 
b_{n-1}^{M_1} \sum_{k=0}^{p-1} (c_1+k)^{m_1}  \cdots  (c_{n-2}+k)^{m_{n-2}} (1+k)^{m_{n-1}}k^{m_n}, 
\end{split}
\end{equation*}
where $M_1=m_1+m_2+\cdots+m_{n}-(p-1)$, and $c_i=b_{n-1}^{-1}b_{i}$ for each $i=1, \, 2, \ \ldots, \, n-2$.
\end{corollary}

We want to find alternative (possibly simpler) expression for the summation $J$ given in \thmref{thm general power1}. 
Doing various computational experiments in \cite{MMA} guided us to interpret this summation in some other ways. First, we provide an example:

Let $p=11$. We obtain the following results by doing symbolic computations for $a$ and $b$ in \cite{MMA}:
$$
\begin{array}{cccc}
s & p-1-s & 2(p-1)-s &\displaystyle{\sum_{k=0}^{p-1} (a+k)^7 (b+k)^7 k^s  \quad \text{modulo $p$}}         \\
\hline
 1 \quad & 9 \quad & 19 \quad & a^5+8 a^4 b+2 a^3 b^2+2 a^2 b^3+8 a b^4+b^5 \\
 2 \quad & 8 \quad & 18 \quad & 4 a^6+7 a^5 b+2 a^4 b^2+7 a^3 b^3+2 a^2 b^4+7 a b^5+4 b^6 \\
 3 \quad & 7 \quad & 17 \quad & -a^7+6 a^6 b+10 a^5 b^2+7 a^4 b^3+7 a^3 b^4+10 a^2 b^5+6 a b^6- b^7 \\
 4 \quad & 6 \quad & 16 \quad & 4 a^7 b+7 a^6 b^2+2 a^5 b^3+7 a^4 b^4+2 a^3 b^5+7 a^2 b^6+4 a b^7 \\
 5 \quad & 5 \quad & 15 \quad & a^7 b^2+8 a^6 b^3+2 a^5 b^4+2 a^4 b^5+8 a^3 b^6+a^2 b^7 \\
 6 \quad & 4 \quad & 14 \quad & \mathbf{-1}+9 a^7 b^3+8 a^6 b^4+10 a^5 b^5+8 a^4 b^6+9 a^3 b^7 \\
 7 \quad & 3 \quad & 13 \quad & \mathbf{4 a+4 b}+9 a^7 b^4+7 a^6 b^5+7 a^5 b^6+9 a^4 b^7 \\
 8 \quad & 2 \quad & 12 \quad & \mathbf{a^2+6 a b+b^2}+a^7 b^5+6 a^6 b^6+a^5 b^7 \\
 9 \quad & 1 \quad & 11 \quad & \mathbf{9 a^3+7 a^2 b+7 a b^2+9 b^3} + 4 a^7 b^6+4 a^6 b^7\\
 10 \quad & 0 \quad & 10 \quad & \mathbf{9 a^4+8 a^3 b+10 a^2 b^2+8 a b^3+9 b^4} +10 a^7 b^7\\
\end{array}
$$
Moreover,
$$
\begin{array}{cccc}
s & p-1-s & 2(p-1)-s &\displaystyle{\sum_{k=0}^{p-1} (a+k)^6 (b+k)^9 k^s  \quad \text{modulo $p$}}         \\
\hline
 1 \quad & 9 \quad & 19 \quad & 10 a^6+a^5 b+10 a^4 b^2+3 a^3 b^3+2 a^2 b^4+3 a b^5+4 b^6 \\
 2 \quad & 8 \quad & 18 \quad & 2 a^6 b+4 a^5 b^2+5 a^4 b^3+10 a^3 b^4+2 a^2 b^5+2 a b^6+8 b^7 \\
 3 \quad & 7 \quad & 17 \quad & 8 a^6 b^2+2 a^5 b^3+2 a^4 b^4+10 a^3 b^5+5 a^2 b^6+4 a b^7+2 b^8 \\
 4 \quad & 6 \quad & 16 \quad & 4 a^6 b^3+3 a^5 b^4+2 a^4 b^5+3 a^3 b^6+10 a^2 b^7+a b^8+10 b^9 \\
 5 \quad & 5 \quad & 15 \quad & \mathbf{-1}+ 6 a^6 b^4+3 a^5 b^5+5 a^4 b^6+6 a^3 b^7+8 a^2 b^8+5 a b^9\\
 6 \quad & 4 \quad & 14 \quad & \mathbf{5 a+2 b}+6 a^6 b^5+2 a^5 b^6+10 a^4 b^7+7 a^3 b^8+7 a^2 b^9 \\
 7 \quad & 3 \quad & 13 \quad & \mathbf{7 a^2+a b+8 b^2}+4a^6 b^6+4 a^5 b^7+8 a^4 b^8+2 a^3 b^9 \\
 8 \quad & 2 \quad & 12 \quad & \mathbf{2 a^3+8 a^2 b+4 a b^2+4 b^3}+8 a^6 b^7+a^5 b^8+7 a^4 b^9 \\
 9 \quad & 1 \quad & 11 \quad & \mathbf{7 a^4+7 a^3 b+10 a^2 b^2+2 a b^3+6 b^4}+2a^6 b^8+5 a^5 b^9 \\
 10 \quad & 0 \quad & 10 \quad & \mathbf{5 a^5+8 a^4 b+6 a^3 b^2+5 a^2 b^3+3 a b^4+6 b^5 }+10 a^6 b^9\\
\end{array}
$$
On the other hand, when $a=b=1$ we have the following equality which is the Vandermonde's identity: 
\begin{equation*}
\begin{split}
\sum_{j=0}^{M}\binom{m}{M-j}\binom{n}{j}a^{M-j}b^j =\sum_{j=0}^{M}\binom{m}{M-j}\binom{n}{j}=\binom{m+n}{M}.
\end{split}
\end{equation*}
Recall that the algebraic proof of Vandermonde's identity relies on the comparison of the coefficients of $(1+x)^m(1+x)^n$ and $(1+x)^{m+n}$.

Having these special cases and the computational results, we ask the following question: What is the relation between the computations given below and in the previous page?

For $p=11$, we have
$$
\begin{array}{cc}
j & \text{Coefficient of $x^j$ in $-(a+x)^7 (b+x)^7$ modulo $p$}         \\
\hline
0 \quad & 10 a^7 b^7 \\
1 \quad &  4 a^7 b^6+4 a^6 b^7 \\
2 \quad &  a^7 b^5+6 a^6 b^6+a^5 b^7 \\
3 \quad &  9 a^7 b^4+7 a^6 b^5+7 a^5 b^6+9 a^4 b^7 \\
4 \quad &  9 a^7 b^3+8 a^6 b^4+10 a^5 b^5+8 a^4 b^6+9 a^3 b^7 \\
5 \quad &  a^7 b^2+8 a^6 b^3+2 a^5 b^4+2 a^4 b^5+8 a^3 b^6+a^2 b^7 \\
6 \quad & 4 a^7 b+7 a^6 b^2+2 a^5 b^3+7 a^4 b^4+2 a^3 b^5+7 a^2 b^6+4 a b^7 \\
7 \quad &  - a^7+6 a^6 b+10 a^5 b^2+7 a^4 b^3+7 a^3 b^4+10 a^2 b^5+6 a b^6- b^7 \\
8 \quad &  4 a^6+7 a^5 b+2 a^4 b^2+7 a^3 b^3+2 a^2 b^4+7 a b^5+4 b^6 \\
9 \quad &  a^5+8 a^4 b+2 a^3 b^2+2 a^2 b^3+8 a b^4+b^5 \\
10 \quad &  9 a^4+8 a^3 b+10 a^2 b^2+8 a b^3+9 b^4 \\
11 \quad &  9 a^3+7 a^2 b+7 a b^2+9 b^3 \\
12 \quad &  a^2+6 a b+b^2 \\
13 \quad &  4 a+4 b \\
14 \quad &  -1 \\
\end{array}
$$
and
$$
\begin{array}{cc}
j & \text{Coefficient of $x^j$ in $-(a+x)^6 (b+x)^9$ modulo $p$}         \\
\hline
0  \quad & 10 a^6 b^9 \\
1  \quad & 2 a^6 b^8+5 a^5 b^9 \\
2  \quad & 8 a^6 b^7+a^5 b^8+7 a^4 b^9 \\
3  \quad & 4 a^6 b^6+4 a^5 b^7+8 a^4 b^8+2 a^3 b^9 \\
4  \quad & 6 a^6 b^5+2 a^5 b^6+10 a^4 b^7+7 a^3 b^8+7 a^2 b^9 \\
5  \quad & 6 a^6 b^4+3 a^5 b^5+5 a^4 b^6+6 a^3 b^7+8 a^2 b^8+5 a b^9 \\
6  \quad & 4 a^6 b^3+3 a^5 b^4+2 a^4 b^5+3 a^3 b^6+10 a^2 b^7+a b^8+10 b^9 \\
7  \quad & 8 a^6 b^2+2 a^5 b^3+2 a^4 b^4+10 a^3 b^5+5 a^2 b^6+4 a b^7+2 b^8 \\
8  \quad & 2 a^6 b+4 a^5 b^2+5 a^4 b^3+10 a^3 b^4+2 a^2 b^5+2 a b^6+8 b^7 \\
9  \quad & 10 a^6+a^5 b+10 a^4 b^2+3 a^3 b^3+2 a^2 b^4+3 a b^5+4 b^6 \\
10  \quad & 5 a^5+8 a^4 b+6 a^3 b^2+5 a^2 b^3+3 a b^4+6 b^5 \\
11 \quad & 7 a^4+7 a^3 b+10 a^2 b^2+2 a b^3+6 b^4 \\
12  \quad & 2 a^3+8 a^2 b+4 a b^2+4 b^3 \\
13  \quad & 7 a^2+a b+8 b^2 \\
14  \quad & 5 a+2 b \\
15  \quad & -1 \\
\end{array}
$$
Since the coefficient of $x^j$ in $-(a+x)^m (b+x)^n$ is $0$ when $j>m+n$, from these computations we expect to have the following relation: 

$\displaystyle{\sum_{k=0}^{p-1} (a+k)^m (b+k)^n k^s}$ is equal to the sum of the coefficients of $x^{p-1-s}$ and $x^{2(p-1)-s}$ in the polynomial $-(a+x)^m (b+x)^n$.

Similarly, for the general case when $b_1=b_2=\cdots=b_{n-1}=1$, the inner sum in the summation $J$ given in \thmref{thm general power1} can be simplified as follows: 
\begin{equation*}
\begin{split}
\sum_{\substack{j_1+\cdots+j_{n-1}=M_i \\ 0 \leq j_1 \leq m_1, \, \ldots, \, 0 \leq j_{n-1} \leq m_{n-1}}}
\binom{m_1}{j_1}  \cdots \binom{m_{n-1}}{j_{n-1}}
=\binom{m_1+\cdots+m_{n-1}}{M_i},
\end{split}
\end{equation*}
This equality is nothing but the generalized Vandermonde's identity and its algebraic proof is given by comparing the coefficients of $(1+x)^{m_1} \cdots (1+x)^{m_{n-1}}$ and $(1+x)^{m_1+\cdots+m_{n-1}}$. 

Again, our computations in \cite{MMA} lead to the following observation: 

The summation $J$ in \thmref{thm general power1} is equal to the sum of the coefficients of 
$x^{i(p-1)-m_n}$ in the polynomial $F(x)=-(b_1+x)^{m_1} \cdots (b_{n-1}+x)^{m_{n-1}}$,  where $i$ is a positive integer. 
We only need the cases with  $i(p-1)-m_n \leq m_1+\cdots+m_{n-1}$, i.e., $M_i \geq 0$ for $M_i$ as in \thmref{thm general power1}. Therefore, we can state the following result:
\begin{theorem}\label{thm general power2}
Given a prime number $p$, let $m_{i}$ and $a_{i}$ be integers with $1 \leq m_i \leq p-1$ and 
$0 \leq a_i \leq p-1$ for each $i=1, \, 2, \, \ldots, \, n $. If $a_i \neq a_j$ when $i \neq j$, we have the following congruence modulo $p$:
\begin{equation*}
\begin{split}
\sum_{k=0}^{p-1} (a_1+k)^{m_1}  \cdots  (a_n+k)^{m_n} &\equiv
-\sum_{i \geq 1} [x^{i(p-1)-m_n}] (b_1+x)^{m_1} \cdots (b_{n-1}+x)^{m_{n-1}},
\end{split}
\end{equation*}
where $[x^j]p(x)$ means the coefficient of $x^j$ in the polynomial $p(x)$, and $b_{i'}=a_{i'}-a_n$ for each index $i'=1, \, 2, \, \ldots, \, n-1$.
\end{theorem}
\begin{proof}
Let $n=1$. we have either $p-1-m_1 = 0$ or  $p-1-m_1 > 0$.

When  $p-1-m_1=0$,  $\displaystyle{\sum_{i \geq 1} [x^{i(p-1)-m_1}]1=[x^{p-1-m_1}]1=1}$. On the other hand, \thmref{thm power sum0} and \remref{rem sum1} give $\displaystyle{\sum_{k=0}^{p-1} (a_1+k)^{m_1}\equiv -1}$.

When  $p-1-m_1>0$, $\displaystyle{\sum_{i \geq 1} [x^{i(p-1)-m_1}]1=0}$, and $\displaystyle{\sum_{k=0}^{p-1} (a_1+k)^{m_1}\equiv 0}$ by \thmref{thm power sum0} and \remref{rem sum1}.

Thus, the theorem agrees with Wolstenholme's Theorem when $n=1$.

Let $n=2$. we have either $2(p-1)-m_2 = m_1$ or  $2(p-1)-m_2 > m_1$. 

The first case happens iff $m_1=m_2=p-1$. In this case,
\begin{equation*}
\begin{split}
\sum_{i \geq 1} [x^{i(p-1)-m_2}] (b_1+x)^{m_1} &= [x^{(p-1)-m_2}] (b_1+x)^{m_1} + [x^{2(p-1)-m_2}] (b_1+x)^{m_1} \\
& =  [x^{0}] (b_1+x)^{p-1} + [x^{p-1}] (b_1+x)^{p-1}\\
&=b_1^{p-1}+1=2, \quad \text{by \eqref{eqn FLT}}.
\end{split}
\end{equation*}
This is consistent with our finding in \thmref{thmcor new power sum3 second}. Namely, $\displaystyle{\sum_{k=0}^{p-1} (a_1+k)^{m_1}(a_2+k)^{m_2} \equiv -2}$.

When $2(p-1)-m_2 > m_1$, $ [x^{2(p-1)-m_2}] (b_1+x)^{m_1}=0$, so
\begin{equation*}
\begin{split}
\sum_{i \geq 1} [x^{i(p-1)-m_2}] (b_1+x)^{m_1} &= [x^{(p-1)-m_2}] (b_1+x)^{m_1} = \binom{m_1}{p-1-m_2}b_1^{m_1-(p-1)+m_2} \\
&\equiv \binom{m_1}{p-1-m_2}b_1^{m_1+m_2} = (a_1-a_2)^{m_1+m_2} \binom{m_1}{m_1+m_2-(p-1)}.
\end{split}
\end{equation*}
Again, this is consistent with our finding in \thmref{thmcor new power sum3 second}.
Note that when $m_1+m_2<p-1$, we have $[x^{(p-1)-m_2}] (b_1+x)^{m_1}=0$ and $\binom{m_1}{m_1+m_2-(p-1)}=0$.

Let $n=3$. We have either $3(p-1)-m_3 = m_1+m_2$ or  $3(p-1)-m_3 > m_1+m_2$.
The first case happens iff $m_1=m_2=m_3=p-1$. In this case,
\begin{equation*}
\begin{split}
\sum_{i \geq 1} [x^{i(p-1)-m_3}] (b_1+x)^{m_1}(b_2+x)^{m_2} &= [x^{0}] (b_1+x)^{m_1}(b_2+x)^{m_2}  + [x^{p-1}] (b_1+x)^{m_1}(b_2+x)^{m_2} \\
& \quad + [x^{2(p-1)}] (b_1+x)^{m_1}(b_2+x)^{m_2}\\
& = b_1^{p-1} b_2^{p-1}+ \sum_{k=0}^{p-1} \binom{m_1}{p-1-k} \binom{m_2}{k} b_1^{p-1-k} b_2^{k}  +1\\
& \equiv 2+\sum_{k=0}^{p-1} (-1)^{p-1-k} \binom{p-1}{k} b_1^{p-1-k} b_2^{k}, \quad \text{by \eqref{eqn FLT}}\\
& \equiv 2+(b_1-b_2)^{p-1}, \quad \text{by Binomial Theorem}\\
& \equiv 3 , \quad \text{by \eqref{eqn FLT}}.
\end{split}
\end{equation*}
This is consistent with our finding in \thmref{thm triple power2sg} and \remref{rem sums}. Namely, in this case we have
$\displaystyle{\sum_{k=0}^{p-1} (a_1+k)^{m_1}(a_2+k)^{m_2} (a_3+k)^{m_3} \equiv -3}$.

When $3(p-1)-m_3 > m_1+m_2$, $[x^{3(p-1)-m_3}] (b_1+x)^{m_1} (b_2+x)^{m_2}=0$, so
\begin{equation*}
\begin{split}
\sum_{i \geq 1} [x^{i(p-1)-m_3}] (b_1+x)^{m_1}(b_2+x)^{m_2} &= [x^{p-1-m_3}] (b_1+x)^{m_1}(b_2+x)^{m_2} \\
& \quad + [x^{2(p-1)-m_3}] (b_1+x)^{m_1}(b_2+x)^{m_2} \\
\end{split}
\end{equation*}
Moreover,
\begin{equation*}
\begin{split}
[x^{p-1-m_3}] (b_1+x)^{m_1}(b_2+x)^{m_2}=
\sum_{\substack{j_1 + j_2 = M_1 \\ j_1 \geq 0, \, j_2 \geq 0}} 
\binom{m_1}{j_1} \binom{m_2}{j_2} b_1^{j_1} b_2^{j_2},
\end{split}
\end{equation*}
where $M_1:=m_1+m_2+m_3-(p-1)$. As we have $m_1-j_1+m_2-j_2=p-1-m_3$ iff $j_1+j_2 = M_1$. 
\begin{equation*}
\begin{split}
[x^{2(p-1)-m_3}] (b_1+x)^{m_1}(b_2+x)^{m_2}=\sum_{\substack{j_1+j_2=M_2\\ j_1 \geq 0, \, j_2 \geq 0}} \binom{m_1}{j_1} \binom{m_2}{j_2} b_1^{j_1} b_2^{j_2},
\end{split}
\end{equation*}
where $M_2:=m_1+m_2+m_3-2(p-1)$. Because, $m_1-j_1+m_2-j_2=2(p-1)-m_3$ iff $j_1+j_2=M_2$.

Therefore, this is consistent with our finding in \thmref{thm general power1}.

When $n>3$, the proof is obtained by following a similar argument.
\end{proof}

It is well known that the coefficients of a polynomial can be obtained in terms of its roots by using Elementary Symmetric Polynomials. Recall that the $r$-th Elementary Symmetric Polynomial in $N$ variables $e_r(x_1,x_2,\ldots,x_N)$ is
defined as the product of all possible $r$ distinct variables among $x_1, \, x_2, \, \ldots, \, x_N$. Namely, we have

$e_0(x_1,x_2,\ldots,x_N)=1$, $e_1(x_1,x_2,\ldots,x_N)=x_1+\cdots+x_N$, 

$\displaystyle{e_2(x_1,x_2,\ldots,x_N)=\sum_{1\leq j_1<j_2\leq N} x_{j_1} x_{j_2}}, \qquad$
$\displaystyle{e_r(x_1,x_2,\ldots,x_N)=\sum_{1\leq j_1< \cdots <j_r \leq N} x_{j_1} \cdots x_{j_r}}$, 

$e_N(x_1,x_2,\ldots,x_N)= x_{1} x_{2} \cdots x_N$, and $e_r(x_1,x_2,\ldots,x_N)=0$ if $r>N$.

Since $[x^{N-j}]p(x)=(-1)^j e_{j}(x_1,\ldots,x_N)$ for a polynomial of degree $N$ with roots $x_1,\ldots,x_N$, we can rewrite 
\thmref{thm general power2} as follows:
\begin{theorem}\label{thm general power3}
Given a prime number $p$, let $m_{i}$ and $a_{i}$ be integers with $1 \leq m_i \leq p-1$ and 
$0 \leq a_i \leq p-1$ for each $i=1, \, 2, \, \ldots, \, n $. If $a_i \neq a_j$ when $i \neq j$, we have the following congruence modulo $p$:
\begin{equation*}
\begin{split}
\sum_{k=0}^{p-1} (a_1+k)^{m_1}  \cdots  (a_n+k)^{m_n} &\equiv
-\sum_{i \geq 1} (-1)^{M_i} e_{M_i}(x_1,x_2,\ldots,x_N),
\end{split}
\end{equation*}
where $x_1, \, \cdots, \, x_N$ are the roots of the polynomial $(b_1+x)^{m_1} \cdots (b_{n-1}+x)^{m_{n-1}}$ with $b_{i'}=a_{i'}-a_n$ for each index $i'=1, \, 2, \, \ldots, \, n-1$, and $M_i=m_1+m_2+\cdots+m_n-i(p-1)$.
\end{theorem}
\begin{proof}
The polynomial $p(x)=(b_1+x)^{m_1} \cdots (b_{n-1}+x)^{m_{n-1}}$ has degree $N=m_1+\cdots+m_{n-1}$, so $i(p-1)-m_n=N-j$ gives
$j=m_1+m_2+\cdots+m_n-i(p-1)$, i.e., $j=M_i$ with our earlier notation. Thus, 
$[x^{i(p-1)-m_n}] p(x)=(-1)^{M_i} e_{M_i}(x_1,x_2,\ldots,x_N)$. 
Hence the proof follows from \thmref{thm general power2}.
\end{proof}
We also recall that Elementary Symmetric Polynomials can be computed by using Newton's Identities (see \cite[Pg. 5]{BE} and 
\cite[Pg. 278]{RS}):
\begin{equation}\label{eqn Newton}
\begin{split}
r e_{r}(x_1,x_2,\ldots,x_N)=\sum_{i=1}^r (-1)^{i-1} e_{r-i}(x_1,\ldots,x_N)p_{i}(x_1,\ldots,x_N),
\end{split}
\end{equation}
where $\displaystyle{p_{r}(x_1,\ldots,x_N)=\sum_{i=1}^N}x_i^r$ is the $r-$th power sum for each integer $r \geq 1$. In our case,
by taking $x_i$'s as the roots of $(b_1+x)^{m_1} \cdots (b_{n-1}+x)^{m_{n-1}}$, we have
\begin{equation}\label{eqn power1}
\begin{split}
(-1)^r p_{r}(x_1,\ldots,x_N)=m_1 b_1^r+m_2 b_2^r+\cdots+m_{n-1}b_{n-1}^{r}.
\end{split}
\end{equation}
Thus, we can also use \eqref{eqn Newton} and \eqref{eqn power1} to compute the desired Elementary Symmetric Polynomials.

\thmref{thm general power1}, \thmref{thm general power2} and \thmref{thm general power3} should be considered along with a generalization of \remref{rem triple all}.

\textbf{Example:} Let $p=17$, and let $I:=\sum_{k=0}^{16}(14+k)^3(10+k)^8(k+4)^9$. In this case, we have $M=3+8+9-16=4$, $p-1-s=7$ and by \corref{cor general constant},
$$
I \equiv \sum_{k=0}^{16}(10+k)^3(6+k)^8k^9 \equiv 6^4 \sum_{k=0}^{16}(6^{-1}10+k)^3(1+k)^8k^9  
\equiv 4 \sum_{k=0}^{16}(13+k)^3(1+k)^8k^9. 
$$
Thus, $I \equiv -4 [x^7](13+x)^3(1+x)^8$ by \thmref{thm general power2} or equivalently $I \equiv -4 (-1)^4 e_4$
by \thmref{thm general power3}. Since $p_r=3 (-13)^r+8 (-1)^r$ by \eqref{eqn power1}, we have
\begin{equation*}
\begin{split}
p_r \equiv
&\begin{cases}
11,  \quad \text{if $r \equiv 0 \, \, mod \, \, 4$}\\
4,  \quad \text{if $r \equiv 1 \, \, mod \, \, 4$}\\
5,  \quad \text{if $r \equiv 2 \, \, mod \, \, 4$}\\
-3,  \quad \text{if $r \equiv 3 \, \, mod \, \, 4$}.
\end{cases}
\end{split}
\end{equation*}
Then, \eqref{eqn Newton} gives
\begin{equation*}
\begin{split}
&e_0 \equiv 1, \quad e_1 \equiv p_1 \equiv 4.\\
&2e_2 = e_1 p_1 -p_2 \equiv 4^2-5, \quad \text{so $e_2 \equiv -3$}.\\
&3e_3 = e_2 p_1 -e_1p_2 +p_3 \equiv (-3)4-4 \cdot 5+(-3), \quad \text{so $e_3 \equiv -6$}.\\
&4e_4 = e_3 p_1 -e_2p_2 +e_1p_3-p_4 \equiv (-6)4-(-3) 5+4(-3)-11, \quad \text{so $e_4 \equiv 9$}.\\
\end{split}
\end{equation*}
Hence, $I \equiv -4 \cdot 9 \equiv  15 \, \, mod \, \, 17$.


\textbf{Declaration of competing interest:} The author declares that he has no known competing financial interest or personal relationship that could have appeared to influence the work reported in this paper.

\end{document}